\documentclass[12pt]{amsart}
\usepackage{amsfonts, amsbsy, amsmath, amssymb}
\usepackage{xcolor,float}
\usepackage{placeins}
\usepackage{caption}
\usepackage{capt-of}
\usepackage{cleveref}
\usepackage{tikz,graphicx} % Required for inserting images
\usetikzlibrary{positioning,arrows.meta,bending,decorations.markings}
\usepackage{graphicx}
\newtheorem{thm}{Theorem}[section]
\newtheorem{lem}[thm]{Lemma}
\newtheorem{cor}[thm]{Corollary}
\newtheorem{prop}[thm]{Proposition}
\newtheorem{exmp}[thm]{Example}

\newtheorem{rmk}[thm]{Remark}

\newtheorem{ques}[thm]{Question}
\newtheorem{thm-con}[thm]{Theorem-Conjecture}
\numberwithin{equation}{section}

\theoremstyle{definition}
\newtheorem{defn}[thm]{Definition}

\newcommand{\Z}{\Bbb Z}

\title[Quandles and Quandle Rings]{A further study of quandles and quandle rings}

\begin{document}

\author[G. Churchil]{Gregory Churchill} 
\address{Department of Mathematics,
SUNY Oswego, Oswego, NY 13126, USA}
\email{gregory.churchill@oswego.edu}

\author[I. R. Churchill]{Indu Rasika Churchill} 
\address{Department of Mathematics,
SUNY Oswego, Oswego, NY 13126, USA}
\email{indurasika.churchill@oswego.edu}

\author[N. Fernando]{Neranga Fernando} 
\address{Department of Mathematics,
Knox College, Galesburg, IL 61401, USA}
\email{nfernando@knox.edu}

\author[B. Kousik]{Bhitali Kousik}
\address{Department of Mathematical Sciences, Tezpur University, Tezpur, Assam 784028, India}
\email{msp24013@tezu.ac.in}


\begin{abstract}
We investigate core quandles and idempotents in quandle rings of core quandles. We answer several questions on the rank of core quandles and nontrivial idempotents in quandle rings. We also present solutions to two questions raised in a recent paper about non-trivial idempotents in quandle rings $\mathbb{Z}[R_5]$ and $\mathbb{Z}[C_5]$, where $\mathbb{Z}$, $R_5$ and $C_5$ are the ring of integers, the dihedral quandle of order 5, and the commutative quandle of order 5, respectively. We then study units in extended quandle rings of a trivial quandle and the Joyce quandle, and nilpotent elements in extended quandles rings of a trivial quandle, where the ground ring is an integral domain. As a consequence, we show that the quandle ring and the extended quandle ring of a trivial quandle are not nil clean rings. We also explore prime rings and semi-prime rings among quandle rings. We introduce zero-divisor graphs of quandle rings and find an intriguing mirror symmetry among the in-degree and out-degree of the vertices. Moreover, we find a bivariate polynomial in the ring $(\mathbb{Z}_{2n+1}[Q])[X,Y]$ that determines the commutative quandle of order $2n+1$, where $2n+1$ is prime.
\end{abstract}

\maketitle

\tableofcontents

\section{Introduction}\label{S1}

A \textit{quandle} is a set $Q$ with a binary operation $\ast\,:\,Q\times Q\to Q$ satisfying the following axioms: 
\begin{enumerate}
\item [(i)] For all $x\in Q, \, x\ast x=x$, 
\item [(ii)] For all $y\in Q$, the map $\beta_y\,:\,Q\to Q$ defined by $\beta_y(x)=x\ast y$ is invertible, and 
\item [(iii)] For all $x,y,z\in Q$, $(x\ast y)\ast z=(x\ast z)\ast (y\ast z)$.
\end{enumerate}

\vskip 0.1in 

Quandles in general are non-associative algebraic structures introduced independently in the 1980s by David Joyce \cite{DJ-1982} and S. V. Matveev \cite{SVM-1982} with the purpose of constructing knot invariants in the three space and knotted surfaces in four space. The three axioms of a quandle algebraically encode the three Reidemeister moves in knot theory. Joyce and Matveev introduced the notion of fundamental quandle and showed that the fundamental quandle is a complete knot invariant up to orientation reversal and mirror reflection. Precisely, two knots $K_1$ and $K_2$ are equivalent (up to reverse and mirror image) if and only if the fundamental quandles $Q(K_1)$ and $Q(K_2)$ are isomorphic. Since then quandles have been extensively studied by many due to their connection to lie algebras, Hopf algebras, quasigroups, Moufang loops, Frobenius algebras and Yang-Baxter equation, etc. The following are some examples of quandles: 
\begin{exmp}
Any nonempty set $Q$ with the operation $x\ast y=x$ for any $x,y\in Q$ is a quandle, called the \textit{trivial} quandle.    
\end{exmp}
\begin{exmp}
Any group $G$ with the binary operation $x\ast y=yxy^{-1}$ is a quandle, called the \textit{conjugation} quandle. 
\end{exmp}
\begin{exmp}\label{E1.3}
The binary operation $x\ast y=yx^{-1}y$ defines a quandle structure on any group $G$, called the \textit{Core} quandle of $G$, denoted by $\text{Core}(G)$.   
\end{exmp}
\begin{exmp}
Let $G=\mathbb{Z}_n$, where $n>1$ is an integer, in Example~\ref{E1.3}. Then the core quandle of $\mathbb{Z}_n$ whose operation is given by $x\ast y = 2y-x$ is called the \textit{dihedral} quandle of order $n$, and is denoted by $R_n$.    
\end{exmp}
\begin{exmp}\label{E1.5}
Define $x\ast y=(n+1)(x+y)$ for all $x,y\in \mathbb{Z}_{2n+1}$. Then $(\mathbb{Z}_{2n+1},\ast)$ is a commutaitve quandle, whcih is also denoted by $C_{2n+1}$. 
\end{exmp}

A subset of a quandle which is also a quandle under the same binary operation is called a subquandle. Like groups, quandles can also be represented by multiplication tables. The multiplication table of the dihedral quandle of order 3, which is a commutative quandle, is given by 

\[
\begin{tabular}{c | c c c}
$\ast$ & $0$ & $1$ & $2$  \\
\cline{1-4}
$0$& $0$ & $2$ & $1$ \\
$1$& $2$ & $1$ & $0$ \\
$2$& $1$ & $0$ & $2$ \\
\end{tabular}
\]

The second axiom of a quandle says that the right multiplication by an element $y$, $\beta_y$, is a bijection. That is, the columns of the multiplication table of a quandle are permutations. A quandle is \textit{Latin} if for each $a\in Q$, the map $\lambda_a\colon Q\mapsto Q$ defined by $\lambda_a(b)=a\ast b$ is a bijection. Equivalently, $Q$ is Latin if the multiplication table of the quandle is a \textit{Latin square}.

The notions of quandle homomorphisms and isomorphisms are natural. Let $\text{Aut}(Q)$ be the group of all the automorphisms of $Q$. The subgroup of $\text{Aut}(Q)$, generated by the column permutations $\beta_y$, is called the \textit{inner} automorphism group of $Q$, and is denoted by $\text{Inn}(Q)$. The \textit{orbit} of an element $x\in Q$, denoted by $\text{Orb}(x)$, is defined to be 
$$\text{Orb}(x) = \{\phi(x) \mid \phi \in \text{Inn}(Q)\}.$$
A quandle is called  \textit{connected} if it has a single orbit. That is, $Q$ is connected if for all $x, y\in Q$, there exists a map $f\in \text{Inn}(Q)$ that maps $x$ to $y$. A quandle is \textit{faithful} if the mapping $a\mapsto \beta_a$ is an injection from $Q$ to $\text{Inn}(Q)$. 
% \(\text{Orb}(x) = \{\phi(x) \mid \phi \in \text{Inn}(Q)\}\)
A quandle $Q$ is said to be commutative if $x\ast y = y \ast x$ for all $x, y\in Q.$ A commutative quandle is clearly a latin quandle, which implies that the quandle is faithful. 

The rank of a quandle $Q$, denoted by $rk(Q)$, is defined to be the minimal number of elements from $Q$ that generate $Q$. It is easy to see that $rk(Core(\mathbb{Z}))= rk(Core(\mathbb{Z}_n))=2$ for $n>1$. In \cite{BF-2024}, Bardakov and Fedoseev posed the following question:

\begin{ques}\label{Q1.5}
What is the connection between the rank of a group $G$ and the rank of $Core(G)$? Is it true that $rk(Core(G))=rk(G)+1$?
\end{ques}

Motivated by this question, in Section~\ref{S2}, we investigate core quandles of the dihedral group $D_n$, and show that the number of orbits of the core quandle $Core(D_n)$ depends on the parity of $n$. This observation leads to a counterexample to the second part of the Question~\ref{Q1.5}. 

Let $(Q,\ast)$ be a quandle and $\Bbbk$ an associative ring with unity 1. Let $e_x$ be a unique symbol corresponding to each $x\in Q$. Let $\Bbbk[Q]$ be the set of formal expressions of the form $\sum_{x\in Q}\,\alpha_xe_x$, where $\alpha_x\in \Bbbk$ such that all but finitely many $\alpha_x=0$. The set $\Bbbk[Q]$ has a free $\Bbbk$-module structure with basis $\{e_x\,|\,x\in Q\}$ and admits a product given by 
$$\Big(\sum_{x\in Q}\,\alpha_xe_x\Big)\,\Big(\sum_{y\in Q}\,\beta_ye_y\Big)\,=\,\sum_{x,y\,\in Q}\,\alpha_x\beta_y\,e_{x\ast y},$$
where $x, y\in Q$ and $\alpha_x, \beta_y \in \Bbbk$. This turns $\Bbbk[Q]$ into a ring called the \textit{quandle ring} with coefficients in $\Bbbk$. Even though the coefficient ring $\Bbbk$ is associative, the quandle ring $\Bbbk[Q]$ is non-associative when $Q$ is a non-trivial quandle. A Quandle $Q$ can be identified as a subset of the quandle ring $\Bbbk[Q]$ with the mapping $x\mapsto e_x$. Quandle rings were introduced in 2019 by Bardakov, Passi and Singh in \cite{{BPS-2019}}. 

Since the introduction of quandle rings, they have been studied by many researchers for their algebraic properties. In \cite{EFT-2019}, the authors showed that quandle rings are never power-associative when the quandle is non-trivial and $\text{Char}(\Bbbk)\neq 2, 3$. Then they showed that if the quandle $Q$ is a union of a finite orbit $Q_1$ and any quandle $Q_2$, then the quandle ring $\Bbbk[Q]$ is not an integral domain. It was also shown in \cite{EFT-2019} that when the quandle $Q$ is of finite cardinality and $\Bbbk$ is a Noetherian ring, $\Bbbk[Q]$ is a both left and right Noetherian ring. 

The surjective ring homomorphism $\epsilon\colon \Bbbk[Q]\rightarrow \Bbbk$ given by 
$$\epsilon\Big(\sum_{x\in Q}\,\alpha_xe_x\Big)=\sum_{x\in Q}\,\alpha_x$$
is called the \textit{augmentation} map. The kernel $\Delta(Q)$ of $\epsilon$ is a two-sided ideal of $\Bbbk[Q]$, called the \textit{augmentation ideal} of $\Bbbk[Q]$. 

In general, the computation of idempotents is an important area of study in ring theory. Since each quandle element is an idempotent, idempotents in quandle rings seem to be the most natural objects in quandle rings that require a further investigation. Let $Q$ be a quandle and $\Bbbk$ an integral domain with unity. A non-zero element $u\in \Bbbk[Q]$ is called an \textit{idempotent} if $u^2=u$. The set of all idempotents of $\Bbbk[Q]$ is usually denoted by $\mathcal{I}(\Bbbk[Q])$. It is clear that the basis elements of the quandle rings $\{e_x\,|\,x\in Q\}$ are idempotents of $\Bbbk[Q]$, and we refer to them as \textit{trivial idempotents}. A non-trivial idempotent is an element of the quandle ring that is not of the form $e_x$ for any $x\in Q$.

Over the last few years, the study of idempotents in quandle rings has garnered attention due to their applications in knot theory. In \cite{ES-2026}, the authors showed that the idempotents in quandle rings can be used to construct stronger knot invariants. We refer the reader to \cite{BE-2026,BES-2026,BPS-2022,ES-2026,ENSS-2023, Wu-2025} for recent studies on idempotents. In \cite{BE-2026}, the authors posed several important questions regarding idempotents in quandle rings. One of them is the following:

\begin{ques}\label{ques 1.7}
Is it true that for any unital ring $\Bbbk$ without zero divisors and an infintie field $\mathbb{F}$, the quandle ring $\Bbbk[Core(\mathbb{F})]$ has only trivial idempotents? 
\end{ques}

In Section~\ref{S3}, we give two examples of infinite fields $\mathbb{F}$ for which the qundle ring $\Bbbk[Core(\mathbb{F})]$ has non-trivial idempotents, where $\Bbbk$ is a unital ring without zero divisors. 

The following proposition appeared in \cite{BES2-2026} gives a sufficient condition in terms of subquandles of the quandle $Q$ for the quandle ring $\Bbbk[Q]$ to have non-trivial idempotents. 

\begin{prop}\cite[Proposition 12.13]{BES2-2026}\label{prop 5.3}
    Let $Q$ be a quandle and $\Bbbk$ be an integral domain with unity. If $Q$ contains a trivial subquandle of order greater than $1,$ then the quandle ring $\Bbbk[Q]$ has non-trivial idempotents.
\end{prop}

In Section~\ref{S3}, we show that the converse of the proposition is not necessarily true. Motivated by the fact that the idempotents in quandle rings can be used to construct knot invariants, we then investigate the idempotents in quandle rings $\Bbbk[\text{Core}(D_3)]$, where $\Bbbk$ is an integral domain and $D_3$ is the dihedral group of order 6. We find necessary and sufficient conditions for the quandle ring $\Bbbk[\text{Core}(D_3)]$ to have non-trivial idempotents. We also present two examples. 

In \cite{BE-2026}, Bardakov and Elhamdadi found necessary and sufficient conditions for quandle rings $\mathbb{Z}[R_5]$ and $\mathbb{Z}[C_5]$ to have nontrivial idempotents $u$ such that $\epsilon(u)=1$. Their conditions were given in the form of systems of non-linear homogeneous multivariate equations which were left unsolved. In Section~\ref{S4}, we use Gr\"{o}bner basis techniques to solve those systems of equations and determine whether the quandle rings $\mathbb{Z}[R_5]$ and $\mathbb{Z}[C_5]$ have nontrivial idempotents $u$ with $\epsilon(u)=1$. 

The extended quandle ring $S$ of $\Bbbk[Q]$ is defined as $S:=\Bbbk[Q]\oplus \Bbbk e,$ where $e(\notin Q)$ is a symbol that satisfies $e(\sum_{i}^{}\alpha_ix_i)=\sum_{i}^{}\alpha_ix_i=(\sum_{i}^{}\alpha_ix_i)e$ ;\cite[p. 614] {BPS-2022}. 

Let $\Bbbk$ be an integral domain with unity $e.$ Then we have $\Bbbk e\cong \Bbbk$, and therefore $S$ can be expressed as $$S=\left\{\sum_{i}^{}\alpha_ix_i+\gamma \mid \alpha_i,\,\gamma\in \Bbbk,\,x_i\in Q\right\}.$$

In Section~\ref{S5}, we investigate units in extended quandle rings $S$ where the quandle $Q$ is a trivial quandle and the Joyce quandle. We find necessary and sufficient conditions for the extended quandle rings to have units. Moreover, we investigate nilpotent elements in the quandle ring of the trivial quandle. An element $a$ in a ring $R$ is called a nilpotent element if there exists a positive integer $n$ such that $a^n=0$. The smallest such positive integer is called the \textit{index} of the nilpotency. An element $r$ in a ring $R$ is called nil clean if there is an idempotent $u\in R$ and a nilpotent $b\in R$ such that $r=u+b$. A ring is called nil clean if every one of its elements is nil clean. We use our results on nilpotent elements to show that the quandle ring of the trivial quandle is not nil clean, and then show that the extended quandle ring is also not nil clean.

A nonzero element $a$ in a ring $R$ is said to be a left zero-divisor if there exists a nonzero element $b\in R$ such that $ab=0$ in $R$. Right zero-divisors are defined similarly. In the commutative setting, we can drop the adjectives ``left'' and ``right'' and just speak of zero-divors, but for non-commutative rings, a left zero-divisor need not be a right zero-divisor. Zero-divisor graphs of rings have been well-studied in the literature. We refer the reader to the book \cite{AABC-2021}, which is devoted to the interplay between rings and graphs, for an in-depth study of the topic. Even though many authors have studied zero-divisor graphs, most of those studies have considered associative rings. Motivated by the fascinating results on zero-divisor graphs of associative rings, in Section~\ref{S6}, we study zero-divsor graphs of extended quandle rings, which are in general non-associative rings. To the best of our knowledge, zero divisor graphs of quandle rings have not appeared in the literature. We define the zero-divisor graph of an extended quandle ring as follows:

\begin{defn}
Let $Q$ be a quandle and $R[Q]$ the quandle ring, where $R$ is the ground ring. Let $S:=R[Q]\oplus Re$ be the extended quandle ring, where $e(\notin Q)$ is a symbol that satisfies $e(\sum_{i}^{}\alpha_ix_i)=\sum_{i}^{}\alpha_ix_i=(\sum_{i}^{}\alpha_ix_i)e$. Let $Z(S)$ be the set of zero-divisors of $S$. The zero-divisor graph of $S$, denoted by $\Gamma(S)$, is the graph with vertex set $Z(S)$, and two distinct vertices $x$ and $y$ are adjacent if $xy=0$. 
\end{defn}

In the definition above, we can replace $S$ and $Z(S)$ by $R[Q]$ and $Z(R[Q])$, respectively, and define the zero-divisor graph of the quandle ring $R[Q]$ similarly. The zero-divisor graph of $R[Q]$ is clearly a subgraph of the zero-divisor graph of $S$. 

We point out to the reader that when $Q$ is not a commutative quandle, the zero-divisor graph of the extended quandle ring is a directed graph whereas when $\Bbbk$ is a commutative ring and $Q$ is a commutative quandle, the zero-divisor graph of the extended quandle ring is an undirected graph. 

In Section~\ref{S6}, we study the zero-divisor graphs of extended quandle rings $S$, where the quandle is a trivial quandle and the ground ring is a finite field. We also investigate zero-divisor graphs of the extended quandle ring of $\Bbbk[Q]$, where $Q$ is the Joyce quandle and $\Bbbk$ is an integral domain. Among our results on in-degrees and out-degrees of vertices, we observed a remarkable mirror symmetry. Precisely, whenever there exists a vertex  $x$ with in-degree $i$ and out-degree $o$, there also exists a vertex with in-degree $o$ and out-degree $i$, and the sum of these two corresponding elements in the quandle ring is not a zero divisor. In Section~\ref{S6}, we give an explanation to this mirror symmetry. We conclude the section with the zero-divisor graph of the extended quandle ring of the dihedral quandle.

A ring $R$ is prime if the product of any two nonzero ideals of $R$ is nonzero. A ring $R$ is semi-prime if the only ideal of $R$ which squares to zero is the zero ideal. Note that if a ring is prime, then it is semi-prime, but the converse is not necessarily true. In Section~\ref{S7}, we find a sufficient condition on the order of a quandle for the quandle ring $\mathbb{Z}_p[Q]$, where $p$ is a prime number, to be not semi-prime, and therefore not prime. We also present a generalization. 

Let $q$ be a prime power and $\mathbb{F}_q$ be the finite field with $q$ elements. In the area of finite fields, researchers use local permutation polynomials to construct latin squares. In fact, there is a one-to-one correspondence between the local permutation polynomials over $\mathbb{F}_q$ and latin squares of order $q$. Let $n\geq 1$ be an integer such that $2n+1$ is prime. A commutative quandle $Q$ of order $2n+1$ is given by the operation $x\ast y=(n+1)(x+y)$. Motivated by local permutation polynomails and the fact that latin quandles are latin squares, in Section~\ref{S7}, we find a bivariate polynomial $f(x,y)$ with coefficients in the quandle ring $\mathbb{Z}_{2n+1}[Q]$ that can be used to construct the commutative quandle of order $2n+1$.  

Throughout the paper, $\Bbbk$ always denotes an integral domain and $\mathbb{F}_q$ denotes the finite field of order $q,$ where $q$ is a prime power.
 
\section{Core quandles of dihedral groups}\label{S2}

Let $G$ be a group. The Core quandle of $G$, denoted by $\text{Core}(G)$, is the quandle whose operation is defined by $x\ast y=yx^{-1}y$ for all $x,y\in Q$. Let $n>2$ be an integer and $D_n$ be the group of symmetries of a regular $n$-gon whose presentation is given by 
$$\displaystyle{D_n=\langle a, b\,|\,a^n=1,\,b^2=1, b^{-1}ab=a^{-1}\rangle}.$$

In this section, we show that the core quandle of the dihedral group is not connected. In particular, we show that $\text{Core}(D_n)$ has two orbits when $n$ is odd, and has four orbits when $n$ is even. Moreover, we answer a question posed by Bardakov and Fedoseev in \cite{BF-2024}. 

\begin{thm}\label{thm 2.1}
Let $n> 2$ be an integer and $D_n$ be the dihedral group of order $2n$. Then    
\begin{center}
\begin{displaymath}
   \text{The number of orbits in Core}(D_n) = \left\{
     \begin{array}{lr}
       2 &  \text{if} \hspace{0.2cm} n\,\text{is odd},\cr
       4 &  \text{if} \hspace{0.2cm} n\,\text{is even}.
     \end{array}
   \right.
\end{displaymath}
\end{center}
\end{thm}

\begin{proof}
Let $x,y\in D_n$. Then we can write $x$ and $y$ as $x=a^{i_1}b^{j_1}$ and $y=a^{i_2}b^{j_2}$. We divide the proof into two cases: $n$ is odd and $n$ is even. \vskip 0.1in 

\noindent\textbf{Case 1.} $n$ is odd. \vskip 0.1in 

\noindent\textbf{Subcase 1.1.} First we consider the left multiplication by rotations. Let $x=a^i$ and $y=a^jb^k$, where $k\in \{0,1\}$. Then 
$$x\ast y=yx^{-1}y=a^jb^k(a^i)^{-1}a^jb^k =a^jb^ka^{j-i}b^k.$$

When $k=0$, we have $x\ast y=a^{2j-i\pmod{n}}$. Since $n$ is odd, for a fixed $i$, $2j-i$ is a permutation. This means when both $x$ and $y$ are rotations, $x\ast y$ is also a rotation, and therefore, in every row corresponds to a rotation, the right multiplication by a rotation is a permutation. 

When $k=1$, we have $x\ast y=a^i=x$. This means when $x$ is a rotation and $y$ is a reflection, $x\ast y$ gives the rotation $x$. 

We have thus far shown that the multiplication by rotations from the left yields one orbit. Next we consider the multiplication by reflections from the left. 

\noindent\textbf{Subcase 1.2.} Let $x=a^ib$ and $y=a^jb^k$, where $k\in \{0,1\}$. A similar computation yields that 
$$x\ast y = a^jb^{k-1}a^{j-i}b^k.$$

When $k=0$, we have $x\ast y=a^ib=x$. This implies that when $x$ is a reflection and $y$ is a rotation, $x\ast y$ is the reflection $x$. 

When $k=1$, we have $x\ast y=a^{2j-i\pmod{n}}\,b$. Since, for a fixed $i$, $2j-i$ is a permutation, when both $x$ and $y$ are reflections, $x\ast y$ is also a reflection. Therefore, in every row corresponds to a reflection, the right multiplication by a reflection is a permutation. This means the multiplication by reflections from the left yields one orbit. 

Thus we have shown that $\text{Core}(D_n)$ has two orbits when $n$ is odd. Now we consider the case $n$ is even. \vskip 0.1in 

\noindent\textbf{Case 2.} $n$ is even. \vskip 0.1in 

\noindent\textbf{Subcase 2.1.} First, let $x=a^i$ and $y=a^jb^k$, where $k\in \{0,1\}$. 

When $k=1$, we have $x\ast y=a^i=x$. This means in every row corresponds to a rotation, multiplication by a reflection from the right gives the roation $x$. 

When $k=0$, we have $x\ast y=a^{2j-i\pmod{n}}$. In this case, both $x$ and $y$ are rotations. That is, $x\ast y=a^{2j-i\pmod{n}}$. Since $n$ is even, for a fixed $i$, $2j-i$ is not a permutation modulo $n$. Fix $i$ and set $2j_1-i\equiv 2j_2-i\pmod{n}$. Then we have $j_1\equiv j_2\pmod{n/2}$, which says that in every row corresponds to a rotation, multiplication by a rotation from the right is not a permutation rather the resulting rotations repeat every $n/2$ times. 

Now we show that no two consecutive rows correspond to rotations share the same elements. Since $x\ast y=a^i=x$ when $k=1$, we only need to consider the case $k=0$: $x\ast y=a^{2j-i\pmod{n}}$. Since columns of a quandle are permutations, assume that the rotation in the $j$th column in the $i$th row appears in the $k$th column in the $(i+1)$st row. That is, we have 
$$2j-i\equiv 2k-(i+1)\pmod{n}.$$
However, this leads to $2(j-k)+1\equiv 0\pmod{n}$, which is a contradiction because $n$ is even.

Now we show that every other row corresponds to a rotation has the same elements. Since $x\ast y=a^i=x$ when $k=1$, we again only need to consider the case $k=0$: $x\ast y=a^{2j-i\pmod{n}}$. 

The question is if there is a column $k$ in $(i+2)$nd row that contains the element (rotation) in row $i$ and column $j$, $a^{2j-i}$. In other words, is there a $k$ such that 
$$2j-i\equiv 2k-(i+2)\pmod{n}?$$
Solving the above congruence gives us $k\equiv j+1\pmod{n/2}$, which means, in fact, the rotation in the $j$th column in $i$th row (corresponds to a rotation) also appears in the right next column in the $(i+1)$st row (corresponds to a rotation) or in the first column if the rotation is in the $n$th column in the $i$th row (corresponds to a rotation). \vskip 0.1in 

\noindent\textbf{Subcase 2.2.} Let $x=a^ib$ and $y=a^jb^k$, where $k\in \{0,1\}$.\vskip 0.1in 

When $k=0$, we have $x\ast y=a^ib=x$.  

When $k=1$, we have $x\ast y=a^{2j-i\pmod{n}}\,b$. Since $n$ is even, for a fixed $i$, $2j-i$ is not a permutation modulo $n$.

The rest of the proof is similar to that of Subcase 2.1. 

\end{proof}

The following proposition answers Question 4.3 of \cite{BF-2024}. 
\begin{prop}\label{prop 2.2}
    Let $D_n$ be the dihedral group of order $2n,$ where $n>2$. Then $rk(D_n)+2=rk(\text{Core}(D_n)).$

\end{prop}
\begin{proof}
    It is well known that $rk(D_n)=2$ for $n>1.$ We now determine the rank of $\text{Core} (D_n),$ considering the cases where $n$ is odd and $n$ is even seperately. Clearly, $rk(\text{Core}(D_n))\neq 1,$ since every element of a quandle is idempotent, and hence no quandle with more than one element can be generated by a single element.
    \vskip 0.1in 

\noindent\textbf{Case 1.} $n$ is odd. \vskip 0.1in 
It follows from Theorem~\ref{thm 2.1} that $\text{Core}(D_n)$ has two orbits. We denote them as $O_1$ and $O_2$. Since each orbit is a subquandle, any set of elements belonging to the same orbit can not generate $\text{Core}(D_n).$ 

If we choose one element from each orbit, then Subcase 1.1 ( the part $k=1$) and Subcase 1.2 ( the case $k=0$) show that the quandle generated by these two elements will be a trivial subquandle (of order $2$) of $\text{Core}(D_n)$. Therefore $rk(\text{Core}(D_n))\neq 2.$ 

Similarly, suppose we choose two elements $x,\,y\in O_1$ and one element $z\in O_2.$  Then, by Subcases 1.1 and 1.2, the subquandle generated by $\{x,\,y,\,z\}$ is $O_1\cup\{z\}.$ Hence, $\text{Core}(D_n)$ cannot be generated by three elements, and therefore $rk(\text{Core}(D_n))\neq 3.$

Finally, if we choose two elements from $O_1$ and two elements from $O_2$, then the resulting subquandle is the entire quandle $\text{Core}(D_n)=O_1\cup O_2$. Consequently, $rk(\text{Core}(D_n))=4=rk(D_n)+2.$ \vskip 0.1in 

\noindent\textbf{Case 2.} $n$ is even. \vskip 0.1in 
By Theorem~\ref{thm 2.1}, $\text{Core}(D_n)$ has four orbits, namely $O_1,\,O_2,\,O_3,$ and $O_4$. Arguing as in the previous case and carefully examining the quandle operation between elements belonging to different orbits (in particular, the orbits corresponding to rotations and those corresponding to reflections), any generating set consisting of fewer than four elements cannot generate the entire quandle, since at least one orbit will necessarily be omitted. On the other hand, if we choose one element from each orbit, then Theorem~\ref{thm 2.1} (Case~2) implies that the subquandle generated by these four elements is the entire quandle. Consequently, $rk(\text{Core}(D_n)=4=rk(D_n)+2.$

\end{proof}
\begin{rmk}
    Note that Theorem~\ref{thm 2.1} and Proposition~\ref{prop 2.2} do not hold for $n=1,$ as $ rk(\text{Core}(D_1))=rk(\text{Core}(\textbf{Z}_2))=2=rk(D_1)+1$ \cite[Question 4.2]{BES-2026}. 
\end{rmk}

\begin{rmk}
    Note that Theorem~\ref{thm 2.1} and Proposition~\ref{prop 2.2} hold for $n=2$, in which case the dihedral group is abelian. 
\end{rmk}

%\section{More Results on Core Quandles}\label{S4}

%A natural question to ask is whether every quandle is a core quandle over some group $G$. The answer to this question is negative. Because there is exactly one group of order $p$, where $p$ is prime, and it is isomorphic to the cyclic group $(Z_p,+)$. $\text{Core}(Z_p)=$ dihedral quandle of order $p$, but there are $p-2$ connected quandles of order $p$. 

\section{Idempotents in quandle rings $\Bbbk[\text{Core}(G)]$}\label{S3}

There has been considerable interest in determining the idempotents of quandle rings, particularly in understanding whether a given quandle ring admits only trivial idempotents or also non-trivial idempotents. In this section, we address a few open questions related to this topic and determine the idempotents of quandle rings $\Bbbk[\text{Core}(D_3)].$ The following example provides a negative answer to Question~\ref{ques 1.7}.
\begin{exmp}\label{exmp 5.1}
    Let $\overline{\mathbb{F}_p}$ be the algebraic closure of the finite field $\mathbb{F}_p,$ where $p$ is a prime. Then $\overline{\mathbb{F}_p}$ is an infinite additive abelian group. Consider the quandle $Q=\text{Core}(\overline{\mathbb{F}_p}).$ Now, $\mathbb{F}_p=\{1,\,2,\,\dots,\,p\}$ is a subset of $Q,$ which is closed under the quandle operation $a*b=2b-a\pmod{p}.$ Since $Q$ is a kei, therefore $Y=\text{Core}(\mathbb{F}_p)$ is a subquandle (of $Q$) of order $p.$ We take $\Bbbk$ to be any integral domain with unity where $p$ is invertible, for example $\mathbb{R}.$ Then by \cite[Proposition 12.11]{BES2-2026}, it follows that $\mathbb{R}[\text{Core}(\overline{\mathbb{F}_p})]$ has a non- trivial idempotent.
    \end{exmp}
    \begin{rmk}
       In Example~\ref{exmp 5.1}, $\overline{\mathbb{F}_p}$ can be replaced by any infinite field of characteristic $p,$ for example the rational function field, $\mathbb{F}_p(t).$
    \end{rmk}

The following proposition, proved in \cite{BES2-2026}, provides a useful criterion for determining whether a quandle ring admits non-trivial idempotents based on the existence of trivial subquandles.
\begin{prop}\cite[Proposition 12.13]{BES2-2026}\label{prop 5.3}
    Let $Q$ be a quandle and $\Bbbk$ be an integral domain with unity. If $Q$ contains a trivial subquandle of order greater than $1,$ then the quandle ring $\Bbbk[Q]$ has non-trivial idempotents.
\end{prop}

A natural question is whether the converse of the above proposition is true. The following counterexample shows that this is not the case.
\begin{exmp}
We consider the quandle $Q=\text{Core}(G),$ where $G$ is a finite group of odd order. Then it is known that $Q$ is latin \cite[Corollary 5.5]{SB-2025}. Hence every orbit of an element in $Q$ must contain more than one element and is therefore a subquandle of order greater than $1,$ say $n>1.$ It follows from \cite[Proposition 12.11]{BES2-2026} that $\Bbbk[Q]$ has non-trivial idempotents, where $\Bbbk$ is an integral domain with unity and $n$ is a unit in $\Bbbk.$ 

On the other hand, since $Q$ is latin, it does not have any trivial subquandles of order greater than $1.$
\end{exmp}

As observed in the proof of Theorem~\ref{thm 2.1}, the quandle $\text{Core}(D_n),\,n\geq 2$ always contains a trivial subquandle of order $2,$ consisting of one rotation and one reflection. Therefore  by Proposition~\ref{prop 5.3}, $\Bbbk[\text{Core}(D_n)],\,n\geq 2$ always has non-trivial idempotents, where $\Bbbk$ is an integral domain with unity. This naturally leads to the problem of determining these idempotents. In the following, we address this question in the case $n=3.$  \vskip 0.1in 

\noindent

 Let $\,A=\{1,\,\cdots,\,n-1\},\,B=\{n,\,n+1,\,\cdots,\,2n-1\}$ and $e_0,\,e_1,\,\cdots,\,e_{2n-1}$ be elements of $D_n=\langle a, b\,|\,a^n=1,\,b^2=1, b^{-1}ab=a^{-1}\rangle.$ We express these elements in the following form:
 \begin{enumerate}
 \item 
     $e_0=a^0=1.$
     \item 
     If $i\in A,$ then $e_i=a^i.$  
     \item 
     If $i\in B,$ then $e_i=a^{i-n}b.$
 \end{enumerate}
Note that, for $i\in A,\,e_i^{-1}=e_{-i\pmod{n}}$ and for $i\in B,\,e_i^{-1}=e_i,\,e_i^2=1.$ We now examine how the quandle operation $a*b=ba^{-1}b$ acts on the elements of $D_n$. From the proof of Theorem~\ref{thm 2.1}, the following relations hold:
\begin{enumerate}
    \item 
    If $(i,\,j)\in(A,\,B)$ or $(i,\,j)\in(B,\,A),$ then $e_i*e_j=e_i.$
    \item 
    If $i\in A$ and $j\in A,$ then $e_i*e_j=e_{2j-i\pmod{n}}.$
    \item 
    If $i\in B$ and $j\in B,$ then $e_i*e_j=a^{2j-i\pmod{n}}b=e_k,$ where $k-n\equiv 2j-i\pmod{n}.$
\end{enumerate}

 Let $Q$ denotes the quandle $\text{Core}(D_n),\,n\geq 2$ and $E_i=e_i-e_0,\,i\in \{1,\,2,\,\cdots,\,2n-1\}$ be the basis elements of the augmentation ideal $\Delta_\Bbbk(Q)$ of the quandle ring $\Bbbk[Q].$ The following lemma describes the multiplication in the quandle ring among these basis elements.
\begin{lem}\label{lem 5.6}
   Assume that $E_0=0.$ Then the following equalities hold.
   \[
   E_iE_j=\begin{cases}
       E_{\overline{2j-i}}-E_{\overline{-i}}-E_{\overline{2j}},\,&~\
       \text{if}~i\in A,\,j\in A\\
       E_i-E_{\overline{-i}},\,& ~\text{if}~i\in A,\, j\in B\\
       -E_{\overline{2j}},\,&~\text{if}~i\in B,\,j\in A\\
       E_k-E_i,\,&~\text{if}~i\in B,\,j\in B,
       
   \end{cases}
   \]
   where $\overline{r}\equiv r\pmod{n}$ and $k=n+\overline{(2j-i)}.$
   
   In particular
   \[
   E_i^2=\begin{cases}
       E_i-E_{\overline{-i}}-E_{\overline{2i}},\,&~\text{if}~i\in A\\
       E_k-E_i,\,&~\text{if}~i\in B,
   \end{cases}
   \]
   where $k=n+\overline{i}$.
   \end{lem}
   \begin{proof}
      We have
    \[
    \begin{split}
        E_iE_j&=(e_i-e_0)(e_j-e_0)\cr
        &=e_i*e_j-e_i*e_0-e_0*e_j+e_0*e_0\cr
        &=e_i*e_j-e_i^{-1}-e_j^2+e_0.
    \end{split}
    \]
    \textbf{Case 1.} $i\in A.$
    \[
    \begin{split}
        E_iE_j&=a^i*e_j-e_{\overline{-i}}-e_j^2+e_0\cr
        &=\begin{cases}
          a^i*a^j-a^{-i\pmod{n}} -(a^j)^2+e_0,\,&~\text{if}~j\in A\\
          a^i-a^{-i\pmod{n}}-e_0+e_0,\,&~\text{if}~j\in B
        \end{cases}\cr
        &=\begin{cases}
            a^{2j-i\pmod{n}}-a^{-i\pmod{n}}-a^{2j\pmod{n}}+e_0,\,&~\text{if}~j\in A\\
            a^i-a^{-i\pmod{n}} ,\,&~\text{if}~j\in B
        \end{cases}\cr
        &=\begin{cases}
            e_{2j-i\pmod{n}}-e_0-(e_{-i\pmod{n}}-e_0)-(e_{2j\pmod{n}}-e_0),\,&~\text{if}~j\in A\\
            e_i-e_0-(e_{-i\pmod{n}}-e_0) ,\,&~\text{if}~j\in B
        \end{cases}\cr
        &=\begin{cases}
       E_{\overline{2j-i}}-E_{\overline{-i}}-E_{\overline{2j}},\,&~\
       \text{if}~j\in A\\
       E_i-E_{\overline{-i}},\,& ~\text{if}~j\in B.
       \end{cases}
    \end{split}
    \]
    In particular $E_i^2=E_i-E_{\overline{-i}}-E_{\overline{2i}}.$\\
    \textbf{Case 2.} $i\in B.$
     \[
    \begin{split}
        E_iE_j&=\begin{cases}
          e_i-e_i-a^{2j\pmod{n}}+e_0,\,&~\text{if}~j\in A\\
          a^{2j-i\pmod{n}}b-e_i-e_0+e_0,\,&~\text{if}~j\in B
        \end{cases}\cr
        &=\begin{cases}
        -e_{2j\pmod{n}}+e_0,\,&~\text{if}~j\in A\\
             a^{2j-i\pmod{n}}b-e_0-E_i,\,&~\text{if}~j\in B
        \end{cases}\cr
        &=\begin{cases}
        -E_{\overline{2j}},\,&~\text{if}~j\in A\\
             E_k-E_i,\,&~\text{if}~j\in B,
        \end{cases}\cr
    \end{split}
    \]
    where $k=n+\overline{(2j-i)}.$ In particular $E_i^2= E_k-E_i,$ where $k=n+\overline{i}.$
       \end{proof}

Let $Q=\text{Core}(D_3).$ Our computation of the idempotents of $\Bbbk[\text{Core}(D_3)]$ follows the approach developed in \cite{BE-2026}. We write $\Bbbk[\text{Core}(D_3)]=\Bbbk \cdot e_0 + \Delta_{\Bbbk}(\text{Core}(D_3)).$

Let $u\in \Bbbk[\text{Core}(D_3)]$ be an idempotent, and we write $u=\alpha \cdot e_0+\delta$, where $\alpha \in \Bbbk, e_0\in \text{Core}(D_3)$, and $\delta \in \Delta_{\Bbbk}(\text{Core}(D_3)).$ Since $u^2=u$, we obtain

\[
\begin{split}
u^2&=(\alpha \cdot e_0+\delta)^2\cr
&=\alpha^2\cdot e_0+\delta^2 +\alpha (e_0\delta +\delta e_0).
\end{split}
\]

Recall that the ideal $\Delta_{\Bbbk}(\text{Core}(D_3))$ is a two-sided ideal, and therefore $\alpha (e_0\delta +\delta e_0)\in\Delta_{\Bbbk}(\text{Core}(D_3))$. Thus $u\in \Bbbk[\text{Core}(D_3)]$ is an idempotent if and only if $\alpha^2=\alpha$ and $\delta^2 +\alpha (e_0\delta +\delta e_0)=\delta$. From the first equation we have $\alpha = 0$ or $\alpha = 1$. When $\alpha = 0$, $u=\delta$, and that implies $\epsilon(u)=0$. When $\alpha = 1$, $u=e_0+\delta$, and that implies $\epsilon(u)=1$. Thus, to find an idempotent $u$ such that $\epsilon(u)=0$, we must look for an element $u$ in $\Delta_{\Bbbk}(\text{Core}(D_3))$. The following computation follows from Lemma ~\ref{lem 5.6}.
\[
\begin{array}{ll}
     E_1^2=E_1-E_{\overline{-1}}-E_{\overline{2}}=E_1-2E_2,\,E_2^2=E_2-2E_1, & E_3^2=E_4^2=E_5^2=0,\\
    E_1E_2=E_0-E_2-E_1=-E_1-E_2, & E_1E_3=E_1E_4=E_1E_5=E_1-E_2,\\ E_2E_1=E_0-E_{\overline{-2}}-E_{\overline{2}}=-E_1-E_2, &E_2E_3=E_2E_4=E_2E_5=E_2-E_1,\\
    E_3E_1=E_4E_1=E_5E_1=-E_2, & E_3E_2=E_4E_2=E_5E_2=-E_1,\\
    E_3E_4=-E_3+E_5, & E_3E_5=--E_3+E_4,\\
    E_4E_3=-E_4+E_5, & E_4E_5=-E_4+E_3,\\
    E_5E_3=-E_5+E_4, & E_5E_4=-E_5+E_3.
\end{array}
\]
\iffalse
$$E_a^2=-2E_{a^2}+E_a, \,E_{a^2}^2=-2E_a+E_{a^2}$$
$$E_{b}^2=0,\, E_{ab}^2=0, \,E_{a^2b}^2=0$$
$$e_1E_a=E_{a^2},e_1E_{a^2}=E_{a}, e_1E_b=0,e_1E_{ab}=0,e_1E_{a^2b}=0$$
$$E_ae_1=E_{a^2},E_{a^2}e_1=E_{a}, E_be_1=b,E_{ab}e_1=E_{ab},E_{a^2b}e_1=E_{a^2b}$$
$$E_aE_{a^2}=-E_a-E_{a^2}, E_aE_{b}=E_aE_{ab}=E_aE_{a^2b}=E_a-E_{a^2}$$
$$E_{a^2}E_a=-E_a-E_{a^2}, E_{b}E_a= E_{ab}E_a= E_{a^2b}E_a=-E_{a^2}$$
$$E_{a^2}E_b=E_{a^2}E_{ab}=E_{a^2}E_{a^2b}=-E_a+E_{a^2}$$
$$E_bE_{a^2}=E_{ab}E_{a^2}=E_{a^2b}E_{a^2}=-E_a$$
$$E_bE_{ab}=-E_{b}+E_{a^2b},\,E_bE_{a^2b}=-E_{b}+E_{ab},\,E_{ab}E_{b}=-E_{ab}+E_{a^2b}$$
$$E_{a^2b}E_b=E_{ab}-E_{a^2b}, E_{ab}E_{a^2b}=E_{b}-E_{ab}, E_{a^2b}E_{ab}=E_{b}-E_{a^2b}$$
\fi

We first look for idempotents $u\in \Bbbk[\text{Core}(D_3)]$ that satisfy $\epsilon(u)=1$. 

\textbf{Case 1.} $\epsilon(u)=1$. Let $u=e_0+\alpha_1E_1+\alpha_{2}E_{2}+\alpha_3E_3+\alpha_{4}E_{4}+\alpha_{5}E_{5}$. Then by comparing the coefficients on both sides of $u^2=u$ gives us the following result. 

\begin{thm}
An element $u\in \Bbbk[\text{Core}(D_3)]$ is a nontrivial idempotent with $\epsilon(u)=1$ if and only if the following system of equations 
$$\alpha_1^2-2\alpha_2^2+2\alpha_2-2\alpha_1\alpha_2+\alpha_1\alpha_3+\alpha_1\alpha_4+\alpha_1\alpha_5-2\alpha_2\alpha_3-2\alpha_2\alpha_4-2\alpha_2\alpha_5-\alpha_1=0$$
$$-2\alpha_1^2+\alpha_2^2+2\alpha_1-2\alpha_1\alpha_2-2\alpha_1\alpha_3-2\alpha_1\alpha_4-2\alpha_1\alpha_5+\alpha_2\alpha_3+\alpha_2\alpha_4+\alpha_2\alpha_5-\alpha_2=0$$
$$-\alpha_3\alpha_4-\alpha_3\alpha_5+2\alpha_4\alpha_5=0$$
$$2\alpha_3\alpha_5-\alpha_3\alpha_4-\alpha_4\alpha_5=0$$
$$2\alpha_3\alpha_4-\alpha_3\alpha_5-\alpha_4\alpha_5=0$$
has nonzero solutions in $\Bbbk$ in which more than one component is nonzero. 
\end{thm}

\begin{exmp}
Take $\Bbbk=\mathbb{F}_2$. Then the system of equations becomes 
$$\alpha_1^2+\alpha_1\alpha_3+\alpha_1\alpha_4+\alpha_1\alpha_5-\alpha_1=0$$
$$\alpha_2^2+\alpha_2\alpha_3+\alpha_2\alpha_4+\alpha_2\alpha_5-\alpha_2=0$$
$$\alpha_3\alpha_4+\alpha_3\alpha_5=0$$
$$\alpha_3\alpha_4+\alpha_4\alpha_5=0$$
$$\alpha_3\alpha_5+\alpha_4\alpha_5=0$$

Solving this system in characteristic 2 gives us the solutions $(\alpha_1,\alpha_2,\alpha_3,\alpha_4,\alpha_5)$: \\
$$(0,0,0,0,0), (1,0,0,0,0), (0,1,0,0,0),(0,0,1,0,0),(0,0,0,1,0),$$
$$(0,0,0,0,1),(0,0,1,1,1),(1,1,0,0,0)$$
Thus the solutions $u$ to the equation $u^2=u$ are 
$$e_0, e_0+E_1, e_0+E_2, e_0+E_3, e_0+E_4, e_0+E_5, e_0+E_3+E_4+E_5, e_0+E_1+E_2$$
There are 8 idempotents $u$ that satisfy $\epsilon(u)=1$. They are
$$e_0, \,e_1,\, e_2,\, e_3, \, e_4, e_5,\, e_3+e_4+e_5,\, e_0+e_1+e_2.$$
\end{exmp}

\textbf{Case 2.} $\epsilon(u)=0$. Let $u=\alpha_1E_1+\alpha_2E_2+\alpha_3E_3+\alpha_4E_4+\alpha_5E_5$. Then by comparing the coefficients on both sides of $u^2=u$ gives us the following result.
\begin{thm}
An element $u\in \Bbbk[\text{Core}(D_3)]$ is a nonzero nontrivial idempotent with $\epsilon(u)=0$ if and only if the following system of equations 
$$\alpha_1^2-2\alpha_2^2+2\alpha_2-2\alpha_1\alpha_2+\alpha_1\alpha_3+\alpha_1\alpha_4+\alpha_1\alpha_5-2\alpha_2\alpha_3-2\alpha_2\alpha_4-2\alpha_2\alpha_5-\alpha_1=0$$
$$-2\alpha_1^2+\alpha_2^2-2\alpha_1\alpha_2-2\alpha_1\alpha_3-2\alpha_1\alpha_4-2\alpha_1\alpha_5+\alpha_2\alpha_3+\alpha_2\alpha_4+\alpha_2\alpha_5-\alpha_2=0$$
$$-\alpha_3\alpha_4-\alpha_3\alpha_5+2\alpha_4\alpha_5=\alpha_3$$
$$2\alpha_3\alpha_5-\alpha_3\alpha_4-\alpha_4\alpha_5=\alpha_4$$
$$2\alpha_3\alpha_4-\alpha_3\alpha_5-\alpha_4\alpha_5=\alpha_5$$
has nonzero solutions in $\Bbbk$. 
\end{thm}

\begin{exmp}
Take $\Bbbk=\mathbb{F}_2$. Then the system of equations becomes 
$$\alpha_1^2+\alpha_1\alpha_3+\alpha_1\alpha_4+\alpha_1\alpha_5-\alpha_1=0$$
$$\alpha_2^2+\alpha_2\alpha_3+\alpha_2\alpha_4+\alpha_2\alpha_5-\alpha_2=0$$
$$\alpha_3\alpha_4+\alpha_3\alpha_5=\alpha_3$$
$$\alpha_3\alpha_4+\alpha_4\alpha_5=\alpha_4$$
$$\alpha_3\alpha_5+\alpha_4\alpha_5=\alpha_5$$

Solving this system in characteristic 2 gives us the solutions $(\alpha_1,\alpha_2,\alpha_3,\alpha_4,\alpha_5)$: \\
$$(0,0,0,0,0), (1,0,0,0,0), (0,1,0,0,0), (0,0,0,1,1), (0,0,1,0,1), (0,0,1,1,0)$$
$$(0,1,0,1,1), (0,1,1,0,1),(0,1,1,1,0),(1,0,0,1,1),(1,0,1,0,1),(1,0,1,1,0)$$
$$(1,1,0,0,0),(1,1,0,1,1), (1,1,1,0,1), (1,1,1,1,0)$$
Thus the solutions $u$ to the equation $u^2=u$ are 
$$0,\,\, E_4+E_5, \,\,E_3+E_5, \,\,E_3+E_4, \,\, E_2, \,\,E_2+E_4+E_5,\,\, E_2+E_3+E_5, \,\,E_2+E_3+E_4,$$
$$E_1,\,\,E_1+E_4+E_5,\,\,E_1+E_3+E_5,\,\,E_1+E_3+E_4,\,\,E_1+E_2,\,\,E_1+E_2+E_4+E_5,$$
$$E_1+E_2+E_3+E_5,\,\,E_1+E_2+E_3+E_4.$$
There are $16$ nonzero idempotents $u$ that satisfy $\epsilon(u)=0$. They are
$$e_4+e_5,\,\,e_3+e_5,\,\,e_3+e_4,\,\,e_2+e_0,\,\,e_2+e_4+e_5+e_0,\,\,e_2+e_3+e_5+e_0,$$
$$e_2+e_3+e_4+e_0, \,\,e_1+e_0,\,\,e_1+e_4+e_5+e_0,\,\,e_1+e_3+e_5+e_0, \,\,e_1+e_3+e_4+e_0,$$
$$e_1+e_2,\,\,e_1+e_2+e_4+e_5,\,\,e_1+e_2+e_3+e_5,\,\,e_1+e_2+e_3+e_4.$$
\end{exmp}

\section{On two open problems on idempotents}\label{S4}

In this section, we give an answer to an open question in \cite{BE-2026}, and also solve a system of equations which was left unsolved in \cite{BE-2026}. 

In \cite{BE-2026}, the authors investigated non-trivial idempotents in the quandle ring where the quandle is the dihedral quandle of order 5, $R_5$, and the coefficient ring is the ring of integers. The following proposition, which appeared in \cite{BE-2026}, gives necessary and sufficient conditions for the quandle ring $\mathbb{Z}[R_5]$ to have non-trivial idempotents for which the value $\varepsilon$ is equal to $1$. However, the authors left the system of equations unsolved. 

\begin{prop}(\cite[Proporition 4.9]{BE-2026})\label{P4.1}
The quandle ring  $\mathbb{Z}[R_5]$ has a non-trivial idempotent for which the value $\varepsilon$ is equal to $1$ if and only if 
the following system has  integer  solutions in which more than one component are non-zero
$$
\begin{cases}
\alpha_1 = \alpha_3 + \alpha_4 +   \alpha_1^2 - \alpha_3^2 - \alpha_4^2 - \alpha_1\alpha_3  - \alpha_1 \alpha_4  - 2 \alpha_3 \alpha_4,    \\ 
\alpha_2 = \alpha_1 + \alpha_3   - \alpha_1^2 + \alpha_2^2 - \alpha_3^2 - \alpha_1\alpha_2  - 2 \alpha_1 \alpha_3  -  \alpha_2 \alpha_3,    \\ 
\alpha_3 = \alpha_2 + \alpha_4   -  \alpha_2^2 + \alpha_3^2 - \alpha_4^2 - \alpha_2\alpha_3  - 2 \alpha_2 \alpha_4  -  \alpha_3 \alpha_4,    \\ 
\alpha_4 = \alpha_1 + \alpha_2 - \alpha_1^2 - \alpha_2^2 + \alpha_4^2 - 2 \alpha_1\alpha_2  -  \alpha_1 \alpha_4  -  \alpha_2 \alpha_4.
\end{cases}
$$
\end{prop}

In this section, we use the Gr\"{o}bner basis technique to show that the only integer solutions of the system are $$(0,0,0,0), (1,0,0,0), (0,1,0,0), (0,0,1,0), (0,0,0,1),$$ 
and thus there are no non-trivial idempotents in the quandle ring $\mathbb{Z}[R_5]$ for which the value $\varepsilon$ is equal to $1$. 

Let 
$$p_1(\alpha_1,\alpha_2,\alpha_3,\alpha_4)=\alpha_3 + \alpha_4 +   \alpha_1^2 - \alpha_3^2 - \alpha_4^2 - \alpha_1\alpha_3  - \alpha_1 \alpha_4  - 2 \alpha_3 \alpha_4-\alpha_1,$$
$$p_2(\alpha_1,\alpha_2,\alpha_3,\alpha_4)=\alpha_1 + \alpha_3   - \alpha_1^2 + \alpha_2^2 - \alpha_3^2 - \alpha_1\alpha_2  - 2 \alpha_1 \alpha_3  -  \alpha_2 \alpha_3-\alpha_2,$$
$$p_3(\alpha_1,\alpha_2,\alpha_3,\alpha_4)= \alpha_2 + \alpha_4   -  \alpha_2^2 + \alpha_3^2 - \alpha_4^2 - \alpha_2\alpha_3  - 2 \alpha_2 \alpha_4  -  \alpha_3 \alpha_4-\alpha_3,$$
$$p_4(\alpha_1,\alpha_2,\alpha_3,\alpha_4)=\alpha_1 + \alpha_2 - \alpha_1^2 - \alpha_2^2 + \alpha_4^2 - 2 \alpha_1\alpha_2  -  \alpha_1 \alpha_4  -  \alpha_2 \alpha_4-\alpha_4.$$

\vskip 0.1in 

The Gr\"{o}bner basis for the ideal generated by the polynomials $p_1, p_2, p_3$ and $p_4$, computed using Mathematica, is given by 

\begin{equation}\label{EQ4.1}
\begin{cases}
\alpha_4-6\alpha_4^2+5\alpha_4^3,\\ 
-\alpha_4 + 4 \alpha_3 \alpha_4 + \alpha_4^2,\\
-\alpha_3 + \alpha_3^2 + \alpha_4 - \alpha_4^2, 
\\-\alpha_4 + 4 \alpha_2 \alpha_4 + \alpha_4^2,
\\ 4 \alpha_2 \alpha_3 - \alpha_4 + \alpha_4^2, \\
-\alpha_2 + \alpha_2^2 + \alpha_4 - \alpha_4^2,
\\ -\alpha_4 + 4 \alpha_1 \alpha_4 + \alpha_4^2,\\
  4 \alpha_1 \alpha_3 - \alpha_4 + \alpha_4^2,\\
4 \alpha_1 \alpha_2 - \alpha_4 + \alpha_4^2,\\
-\alpha_1 + \alpha_1^2 + \alpha_4 - \alpha_4^2.
\end{cases}
\end{equation}

The first polynomial equation from \eqref{EQ4.1}, $\alpha_4-6\alpha_4^2+5\alpha_4^3=0$, gives us the integer solutions $\alpha_4=0$ or $\alpha_4=1$. When $\alpha_4=1$, the system of equations \eqref{EQ4.1} reduces to

\begin{equation}\label{EQ4.2}
\begin{cases}
4 \alpha_3 =0,\\
-\alpha_3 + \alpha_3^2 =0, 
\\ 4 \alpha_2 =0,
\\ 4 \alpha_2 \alpha_3 =0, \\
-\alpha_2 + \alpha_2^2 =0,
\\4 \alpha_1=0,\\
4 \alpha_1 \alpha_3 =0,\\
4 \alpha_1 \alpha_2 =0,\\
-\alpha_1 + \alpha_1^2 =0.
\end{cases}
\end{equation} 

The only integer solution to the system of equations is $(0,0,0,1)$. 

When $\alpha_4=0$, the system of equations reduces to

\begin{equation}\label{EQ4.3}
\begin{cases}
-\alpha_3 + \alpha_3^2 =0, 
\\ 4 \alpha_2 \alpha_3 =0, \\
-\alpha_2 + \alpha_2^2 =0,\\
4 \alpha_1 \alpha_3 =0,\\
4 \alpha_1 \alpha_2 =0,\\
-\alpha_1 + \alpha_1^2 =0.
\end{cases}
\end{equation} 

Clearly, the only integer solutions when $\alpha_4=0$ are $$(0,0,0,0), (1,0,0,0), (0,1,0,0), (0,0,1,0).$$ 

Thus, the only integer solutions of the system of equations in Proposition~\ref{P4.1} are $$(0,0,0,0), (1,0,0,0), (0,1,0,0), (0,0,1,0), (0,0,0,1).$$ 

\vskip 0.1in 

In \cite{BE-2026}, the authors also investigated idempotents in the quandle ring $\Z[C_5]$, where $C_5$ denotes the commutative quandle of order 5. Suppose now that $u$ is an idempotent in  $\Z[C_5]$ for which $\varepsilon(u) =1$.  Then 
$$
u = a_0 + \delta = a_0 + \beta_1 f_1 + \beta_2 f_2 + \beta_3 f_3 + \beta_4 f_4
$$
for some integers $\beta_i$. Here $f_i=a_i-a_0$, where $a_i\in C_5$ for $i=1,2,3,4$. The authors showed that the element $u$ is an idempotent if and only if $2 \delta a_0 - \delta = -\delta^2$, which is equivalent to the system
$$
\begin{cases}
2 \beta_2 - \beta_1 = -\beta_1^2 + 2 \beta_2^2 + 2 \beta_1 \beta_2 + 2 \beta_2 \beta_3 + 2 \beta_2 \beta_4 - 2 \beta_3 \beta_4,    \\ 
2 \beta_4 - \beta_2 = -\beta_2^2 + 2 \beta_4^2 - 2 \beta_1 \beta_3 + 2 \beta_1 \beta_4 + 2 \beta_2 \beta_4 + 2 \beta_3 \beta_4,    \\ 
2 \beta_1 - \beta_3 = -\beta_3^2 + 2 \beta_1^2 + 2 \beta_1 \beta_2 + 2 \beta_1 \beta_3 + 2 \beta_1 \beta_4 - 2 \beta_2 \beta_4,    \\ 
2 \beta_3 - \beta_4 = -\beta_4^2 + 2 \beta_3^2 - 2 \beta_1 \beta_2 + 2 \beta_1 \beta_3 + 2 \beta_2 \beta_3 + 2 \beta_3 \beta_4,
\end{cases}
$$

and raised the following question: \vskip 0.2in 

\textbf{Question} (\cite[Question 5.10]{BE-2026})
Is it true that this system does not have  integer  solutions in which more than one component is non-zero? \vskip 0.2in

We give a positive answer to the question using the Gr\"{o}bner basis technique and the rational root theorem showing that the only integer solutions of the system are $$(0,0,0,0), (1,0,0,0), (0,1,0,0), (0,0,1,0), (0,0,0,1).$$

Let 
$$q_1(\beta_1, \beta_2, \beta_3, \beta_4)=-\beta_1^2 + 2 \beta_2^2 + 2 \beta_1 \beta_2 + 2 \beta_2 \beta_3 + 2 \beta_2 \beta_4 - 2 \beta_3 \beta_4- 2 \beta_2 + \beta_1 ,$$
$$q_2(\beta_1, \beta_2, \beta_3, \beta_4)=-\beta_2^2 + 2 \beta_4^2 - 2 \beta_1 \beta_3 + 2 \beta_1 \beta_4 + 2 \beta_2 \beta_4 + 2 \beta_3 \beta_4 - 2 \beta_4 + \beta_2,$$
$$q_3(\beta_1, \beta_2, \beta_3, \beta_4)=-\beta_3^2 + 2 \beta_1^2 + 2 \beta_1 \beta_2 + 2 \beta_1 \beta_3 + 2 \beta_1 \beta_4 - 2 \beta_2 \beta_4 - 2 \beta_1 + \beta_3,$$
$$q_4(\beta_1, \beta_2, \beta_3, \beta_4)=-\beta_4^2 + 2 \beta_3^2 - 2 \beta_1 \beta_2 + 2 \beta_1 \beta_3 + 2 \beta_2 \beta_3 + 2 \beta_3 \beta_4 - 2 \beta_3 + \beta_4.$$

The Gr\"{o}bner basis for the ideal generated by the polynomials $q_1, q_2, q_3$ and $q_4$, computed using Mathematica, is given by 

\begin{footnotesize}
\begin{equation}\label{EQ4.4}
\begin{cases}
6 \beta_4 - 21 (\beta_4)^2 - 110 (\beta_4)^3 + 500 (\beta_ 4)^4 - 
 1000 (\beta_4)^5 + 625 (\beta_ 4)^6, \\\\ -6 \beta_ 4 
 + 24 \beta_ 3 \beta_ 4 - 9 (\beta_ 4)^2 + 60 \beta_ 3 (\beta_ 4)^2 + 
 65 (\beta_ 4)^3 - 200 \beta_ 3 (\beta_ 4)^3 - 175 (\beta_ 4)^4 + 
 500 \beta_ 3 (\beta_ 4)^4 + 125 (\beta_ 4)^5, \\\\
-24 \beta_ 3 + 24 (\beta_ 3)^2 + 6 \beta_ 4 + 12 \beta_ 3 \beta_ 4 - 
 60 (\beta_ 3)^2 \beta_ 4 + 129 (\beta_ 4)^2 \\ + 
 20 \beta_ 3 (\beta_ 4)^2 + 100 (\beta_ 3)^2 (\beta_ 4)^2 - 
 485 (\beta_ 4)^3 - 200 \beta_ 3 (\beta_ 4)^3 + 975 (\beta_ 4)^4 - 
 625 (\beta_ 4)^5,\\\\ -192 \beta_ 3 - 128 (\beta_ 3)^2 + 
 320 (\beta_ 3)^3 + 246 \beta_ 4 - 576 \beta_ 3 \beta_ 4 + 
 569 (\beta_ 4)^2 \\+ 1920 \beta_ 3 (\beta_ 4)^2 - 3365 (\beta_ 4)^3 - 
 4800 \beta_ 3 (\beta_ 4)^3 + 8175 (\beta_ 4)^4 - 5625 (\beta_ 4)^5, \\\\
-6 \beta_ 4 + 12 \beta_ 2 \beta_ 4 + 12 \beta_ 3 \beta_ 4 + 
 21 (\beta_ 4)^2 - 30 \beta_ 2 (\beta_ 4)^2 - 
 30 \beta_ 3 (\beta_ 4)^2 - 40 (\beta_ 4)^3 + 
 50 \beta_ 2 (\beta_ 4)^3 + 50 \beta_ 3 (\beta_ 4)^3 + 
 25 (\beta_ 4)^4, \\\\
128 \beta_ 2 \beta_ 3 + 78 \beta_ 4 - 64 \beta_ 2 \beta_ 4 - 
 256 \beta_ 3 \beta_ 4 + 320 (\beta_ 3)^2 \beta_ 4 - 
 643 (\beta_ 4)^2\\ + 320 \beta_ 2 (\beta_ 4)^2 + 
 320 \beta_ 3 (\beta_ 4)^2 + 1815 (\beta_ 4)^3 - 3125 (\beta_ 4)^4 + 
 1875 (\beta_ 4)^5, \\\\
-32 \beta_ 2 + 32 (\beta_ 2)^2 - 48 \beta_ 3 + 48 (\beta_ 3)^2 - 
 14 \beta_ 4 + 80 \beta_ 2 \beta_ 4 + 96 \beta_ 3 \beta_ 4 - 
 80 (\beta_ 3)^2 \beta_ 4 \\+ 499 (\beta_ 4)^2 - 
 80 \beta_ 2 (\beta_ 4)^2 - 80 \beta_ 3 (\beta_ 4)^2 - 
 1735 (\beta_ 4)^3 + 3125 (\beta_ 4)^4 - 1875 (\beta_ 4)^5, \\\\
-64 \beta_ 3 + 64 (\beta_ 3)^2 -142 \beta_ 4 + 
 128 \beta_ 1 \beta_ 4 \\+ 192 \beta_ 2 \beta_ 4 + 
 384 \beta_ 3 \beta_ 4 - 320 (\beta_ 3)^2 \beta_ 4 + 
 707 (\beta_ 4)^2 - 320 \beta_ 2 (\beta_ 4)^2 - 
 320 \beta_ 3 (\beta_ 4)^2 - 1815 (\beta_ 4)^3 + 3125 (\beta_ 4)^4 - 
 1875 (\beta_ 4)^5, \\\\ 32 \beta_ 3 + 128 \beta_ 1 \beta_ 3 - 
 32 (\beta_ 3)^2 + 14 \beta_ 4 - 96 \beta_ 2 \beta_ 4 + 
 64 \beta_ 3 \beta_ 4 \\- 160 (\beta_ 3)^2 \beta_ 4 - 
 419 (\beta_ 4)^2 - 160 \beta_ 2 (\beta_ 4)^2 - 
 160 \beta_ 3 (\beta_ 4)^2 + 1655 (\beta_ 4)^3 - 3125 (\beta_ 4)^4 + 
 1875 (\beta_ 4)^5, \\\\
64 \beta_ 1 \beta_ 2 + 80 \beta_ 3 - 80 (\beta_ 3)^2 + 14 \beta_ 4 - 
 80 \beta_ 2 \beta_ 4 - 160 \beta_ 3 \beta_ 4 + 
 80 (\beta_ 3)^2 \beta_ 4 \\- 499 (\beta_ 4)^2 + 
 80 \beta_ 2 (\beta_ 4)^2 + 80 \beta_ 3 (\beta_ 4)^2 + 
 1735 (\beta_ 4)^3 - 3125 (\beta_ 4)^4 + 1875 (\beta_ 4)^5, \\\\
-64 \beta_ 1 + 64 (\beta_ 1)^2 - 32 \beta_ 3 + 32 (\beta_ 3)^2 + 
 50 \beta_ 4 - 32 \beta_ 2 \beta_ 4 - 64 \beta_ 3 \beta_ 4 + 
 160 (\beta_ 3)^2 \beta_ 4 \\+ 355 (\beta_ 4)^2 + 
 160 \beta_ 2 (\beta_ 4)^2 + 160 \beta_ 3 (\beta_ 4)^2 - 
 1655 (\beta_ 4)^3 + 3125 (\beta_ 4)^4 - 1875 (\beta_ 4)^5.
\end{cases}
\end{equation} 
\end{footnotesize}

We adhere to the same strategy that was used to solve the previous system of equations. Consider the one variable polynomial equation in $\beta_4$ from \eqref{EQ4.4}: $$6 \beta_4 - 21 (\beta_4)^2 - 110 (\beta_4)^3 + 500 (\beta_ 4)^4 - 
 1000 (\beta_4)^5 + 625 (\beta_ 4)^6=0.$$
We have $\beta_4=0$ or $6 - 21\beta_4 - 110 \beta_4^2 + 500 \beta_ 4^3 - 1000 \beta_4^4 + 625 \beta_ 4^5=0$.
The rational root theorem says that the only possible integer solutions to $$6 - 21\beta_4 - 110 \beta_4^2 + 500 \beta_ 4^3 - 1000 \beta_4^4 + 625 \beta_ 4^5=0$$
are $\pm 1, \pm 2, \pm 3$ and $\pm 6$. Straightforward computations yield that the only integer solution is $\beta_4=1$. We have $\beta_4=0$ or $\beta_4=1$. 

When $\beta_4=0$, the system of equations \eqref{EQ4.4} reduces to 

\begin{footnotesize}
\begin{equation}
\begin{cases}
-24 \beta_ 3 + 24 (\beta_ 3)^2 =0,\\\\ 
-192 \beta_ 3 - 128 (\beta_ 3)^2 + 
 320 (\beta_ 3)^3 =0,\\\\
128 \beta_ 2 \beta_ 3 =0, \\\\
-32 \beta_ 2 + 32 (\beta_ 2)^2 - 48 \beta_ 3 + 48 (\beta_ 3)^2=0, \\\\
-64 \beta_ 3 + 64 (\beta_ 3)^2=0, \\\\ 
32 \beta_ 3 + 128 \beta_ 1 \beta_ 3 - 
 32 (\beta_ 3)^2 =0, \\\\
64 \beta_ 1 \beta_ 2 + 80 \beta_ 3 - 80 (\beta_ 3)^2 =0, \\\\
-64 \beta_ 1 + 64 (\beta_ 1)^2 - 32 \beta_ 3 + 32 (\beta_ 3)^2=0.
\end{cases}
\end{equation}
\end{footnotesize}

Clearly, the only integer solutions when $\beta_4=0$ are 
$$(0,0,0,0), (1,0,0,0), (0,1,0,0), (0,0,1,0).$$

When $\beta_4=1$, the second polynomial becomes $384\beta_3$, whose only root is zero. Then the fifth polynomial becomes $32\beta_2$ which implies that $\beta_2=0$ is the only root. It follows from the eighth polynomial that $\beta_1=0$. Thus the only integer solution when $\beta_4=1$ is $(0,0,0,1)$.

We conclude that the only integer solutions of the system are $$(0,0,0,0), (1,0,0,0), (0,1,0,0), (0,0,1,0), (0,0,0,1).$$

\section{Units and nilpotent elements in quandle rings}\label{S5}

We recall the definition of an extended quandle ring. The extended quandle ring $S$ of $\Bbbk[Q]$ is defined as $S:=\Bbbk[Q]\oplus \Bbbk e,$ where $e(\notin Q)$ is a symbol that satisfies $$e(\sum_{i}^{}\alpha_ix_i)=\sum_{i}^{}\alpha_ix_i=(\sum_{i}^{}\alpha_ix_i)e.$$

Let $\Bbbk$ be an integral domain with unity $e.$ Then we have $\Bbbk e\cong \Bbbk$, and therefore $S$ can be expressed as $$S=\left\{\sum_{i}^{}\alpha_ix_i+\gamma \mid \alpha_i,\,\gamma\in \Bbbk,\,x_i\in Q\right\}.$$

In this section, we present some results concerning units and nilpotent elements in quandle rings as well as their extended quandle rings. Since a quandle ring has no unity, it has no units. Hence our discussion  of units is restricted to extended quandle rings. The augmentation map for extended quandle ring, $\epsilon^o:S\mapsto \Bbbk$ is defined as $$\epsilon^o(\sum_{i}^{}\alpha_ix_i+\gamma)=\sum_{i}^{}\alpha_i+\gamma.$$
The following lemma is believed to be well-known, and we omit its proof.

\begin{lem}
    $\epsilon^o$ is a surjective ring homomorphism.
\end{lem}

If we consider $Q$ to be a trivial quandle, which is associative, then $\epsilon^o$ maps units in $S$ to units in $\Bbbk.$ An element $u$ is said to be a unit in $S$ if there exists $v\in S$ such that $uv=e=vu$ and in that case $\epsilon(u)\in U,$ where $U$ is the set of units in $\Bbbk.$ We next discuss the units of $S.$

In \cite{BPS-2019}, the authors studied units in the extended quandle ring of the trivial quandle. However, our approach in the following theorem is different. We give necessary and sufficient conditions for an element to be a unit in the extended quandle ring in terms of the solvability of a system of equations, and we use the results of this approach in Section 6. Therefore, we give a proof of the following theorem as we would like the paper to be self-contained. 

\begin{thm}\label{thm 5.2}
    Let $\Bbbk$ be an integral domain with unity $e,\,U$ be the set of units of $\Bbbk,$ $Q$ be a trivial quandle and $S$ be the extended quandle ring of $\Bbbk[Q].$ An element $u=\alpha x+\delta+\gamma\in S,$ where $\alpha,\gamma\in \Bbbk,\,x\in Q,\,\delta\in\Delta_\Bbbk(Q)$ is a unit if and only if there exists $\alpha_1,\,\gamma_1\in \Bbbk,\,\delta_1\in\Delta_\Bbbk(Q),$ such that $\alpha+\gamma,\,\alpha_1+\gamma_1\in U$ and the following system of equations is solvable over $\Bbbk.$
    \begin{equation}  \label{eqn7.3}
    \begin{array}{rcccl}
    \alpha\alpha_1+\gamma_1\alpha+\gamma\alpha_1&=& 0 & &\\
    \alpha_1\delta+\gamma_1\delta+\gamma\delta_1&=& 0 & &\\
    %&=&\alpha\delta_1+\gamma_1\delta+\delta_1\gamma\\
    \alpha_1\delta &=&\alpha\delta_1 & &\\
    \gamma\gamma_1&=& e & &
    \end{array}
\end{equation}
\end{thm}
\begin{proof}
Any element $u\in S$ can be written as $u=\alpha x+\delta+\gamma,$ where $x\in Q,\,\delta\in\Delta_\Bbbk(Q),\,\alpha,\,\gamma\in R.$ Assume that $u$ is unit. Then there exists $v=\alpha_1 x+\delta_1+\gamma_1,\,x\in Q,\,\delta_1\in\Delta_\Bbbk(Q),\,\alpha_1,\,\gamma_1\in \Bbbk.$ such that 
\begin{equation}\label{eqn 7.1}
    uv=e=vu.
\end{equation}
We now compute $uv$ and $vu.$
\[
\begin{split}
    uv&=(\alpha x+\delta+\gamma)(\alpha_1 x+\delta_1+\gamma_1)\cr
    &=\alpha\alpha_1 x+\alpha x\delta_1+\alpha_1\delta x+\delta\delta_1+\gamma_1\alpha x+\gamma_1\delta+\gamma\alpha_1 x+\gamma\delta_1
\end{split}
\]
When $Q$ is trivial, we know that $\Delta_\Bbbk^2(Q)=\{0\}$ \cite[Theorem 3.5]{BPS-2019}. Therefore, $$uv=\alpha\alpha_1 x+\gamma_1\alpha x+\gamma\alpha_1 x+\alpha x\delta_1+\alpha_1\delta x+\gamma_1\delta+\gamma\delta_1+\gamma\gamma_1$$ Similarly,
$$vu=\alpha_1\alpha x+\gamma\alpha_1 x+\gamma_1\alpha x+\alpha_1 x\delta+\alpha\delta_1 x+\gamma\delta_1+\gamma_1\delta+\gamma_1\gamma$$

Since $S=\Bbbk[Q]\oplus \Bbbk=\Bbbk x\oplus\Delta_\Bbbk(Q)\oplus \Bbbk,$ where $x\in Q$ and $\Delta_\Bbbk(Q)$ is a two sided ideal, equation~\ref{eqn 7.1} gives us the following set of equations.
\begin{equation}\label{eqn}  
    \begin{array}{rcccl}
    \alpha\alpha_1+\gamma_1\alpha+\gamma\alpha_1&=& 0 &=&\alpha_1\alpha+\gamma\alpha_1+\gamma_1\alpha\\
    \alpha x\delta_1+\alpha_1\delta x+\gamma_1\delta+\gamma\delta_1 &=& 0 &=&\alpha_1 x\delta+\alpha\delta_1x+\gamma\delta_1+\delta_1\gamma\\
    \gamma\gamma_1&=& e &=&\gamma_1\gamma.
    \end{array}
\end{equation}
Since $Q$ is trivial, a simple calculation shows that $x \delta= x\delta_1=0$ and $\delta x= \delta,\,\delta_1 x=\delta_1.$ Also using the fact that $\Bbbk$ is commutative,(\ref{eqn}) can be simplified as
\begin{equation*} 
    \begin{array}{rcccl}
    \alpha\alpha_1+\gamma_1\alpha+\gamma\alpha_1&=& 0 & &\\
    \alpha_1\delta+\gamma_1\delta+\gamma\delta_1&=& 0 & &\\
    %&=&\alpha\delta_1+\gamma_1\delta+\delta_1\gamma\\
    \alpha_1\delta &=&\alpha\delta_1 & &\\
    \gamma\gamma_1&=& e & &\\
    \end{array}
\end{equation*}
By definition, we have $\epsilon^o(u)=\alpha+\gamma$ and $\epsilon^o(v)=\alpha_1+\gamma_1.$ If $u$ is a unit (so is $v$), then we must have $\alpha+\gamma,\,\alpha_1+\gamma_1\in U.$
\end{proof}

\begin{exmp}\label{exmp 7.3}
    Let $\Bbbk=\mathbb{F}_2$ and $Q$ be a trivial quandle. Then $U=\{1\}.$ Let $u=\alpha x+\delta+\gamma\in S$ be a unit and $v=\alpha_1x+\delta_1+\gamma_1$ be its inverse. Then $\alpha+\gamma=\alpha_1+\gamma_1=1$. From (\ref{eqn7.3}), it is clear that $(\gamma,\,\gamma_1)=(1,\,1).$ So $\alpha=\alpha_1=0.$ Finally, we have $\gamma_1\delta+\gamma\delta_1=0,$ or $\delta=\delta_1,$ which is solvable over $\mathbb{F}_2$ Therefore, $u$ is a unit in $S$ if and only if $u=\delta+1.$ In particular, when $Q=\{x_0,\,x_1\,x_2\}$  then $$u\in\{1,\,x_0+x_1+1,\,x_0+x_2+1,\,x_1+x_2+1\}.$$
    
\end{exmp}

%We say $x\in R[Q]$ is \textit{nilpotent} of index $k,$ if $k$ is the smallest positive integer such that $x^k=0.$ Clearly, $0$ is a nilpotent element of index $1$ and any element $\delta\in\Delta_\Bbbk(Q)$ is nilpotent of index $2,$ where $Q$ is a trivial quandle.

%\begin{thm}
    %Let $x=\sum_{i}{}\alpha_ix_i\in R[Q],$ where $R$ is an integral domain with unity and $Q$ be a trivial quandle. Then $x$ is  nilpotent of index $k\geq 1$  if and only if $\sum_{i}^{}\alpha_i^{k-1}=0.$ 
%\end{thm}
%\begin{proof}
    %$x^2=\sum_{i}^{}\alpha_i\sum_{j}^{}\alpha_j x_i,\,x^3=\sum_{i}^{}\alpha_i\sum_{j}^{}\alpha_j^2 x_i.$ In general, we have $$x^k=\sum_{i}^{}\alpha_i\sum_{j}^{}\alpha_j^{k-1}x_i,\,k>1.$$ Therefore $x^k=0$ if and only if $\sum_{i}^{}\alpha_i\sum_{j}^{}\alpha_j^{k-1}x_i=0.$ Comparing the coefficients of $x_i,$ we get $\alpha_i\sum_{j}^{}\alpha_j^{k-1}=0,$ for all $i.$ If $\alpha_i=0,$ for all $i,$ then $x=0,$ which is always a nilpotent. So assume that $\alpha_i\neq 0,$ for some $i,$ and it then follows that $\sum_{j}^{}\alpha_j^{k-1}=0.$ 
%\end{proof}
We now characterize the nilpotent elements of $\Bbbk[Q],$ where $Q$ is a trivial quandle. Clearly, $0$ is a nilpotent element of index $1$ and any element $\delta\in\Delta_\Bbbk(Q)$ is nilpotent of index $2.$
\begin{thm}\label{thm 5.4}
    Let $x\in \Bbbk [Q],$ where $\Bbbk$ is an integral with unity and $Q$ is a trivial quandle. Then $x$ is  nilpotent  if and only if $x\in\Delta_\Bbbk(Q),$ where $\Delta_\Bbbk(Q)$ is the augmentation ideal of $\Bbbk[Q].$ Consequently any non-zero nilpotent element of $\Bbbk[Q]$ is of index $2.$
    \end{thm}
    \begin{proof}
        Any element $x\in \Bbbk[Q]$ can be expressed as $x=\alpha x_0+\delta,$ where $\alpha\in \Bbbk,\,x_0\in Q~\text{is fixed}$ and $\delta\in\Delta_\Bbbk(Q).$ Computing the powers of $x$ and using the fact that $\delta^2=0,\,x_0\delta=0$ and $\delta x_0=\delta,$ we get
        \[
        \begin{split}
            x^2&=\alpha^2x_0+\alpha\delta,\cr
            x^3&=\alpha^3x_0+\alpha^2\delta,\cr
            x^4&=\alpha^4x_0+\alpha^3\delta.\cr  
        \end{split}
        \]
        In general, we have $x^k=\alpha^kx_0+\alpha^{k-1}\delta.$ Therefore, $x^k=0$ if and only if $\alpha^k=0=\alpha^{k-1}.$ Since $\Bbbk$ is an integral domain, we must have $\alpha=0,$ yielding $x=\delta\in\Delta_\Bbbk(Q).$ Thus the proof follows.
    \end{proof}
The following proposition presents a result on nilpotent elements of the extended quandle ring.
\begin{prop}\label{prop 5.5}
    Let $a$ be a nilpotent element of the extended quandle ring $S$ of $\Bbbk[Q],$ where $Q$ is a quandle of order $>1.$ Then $a\in \Bbbk[Q].$
    \end{prop}
    \begin{proof}
        Consider an element $a\in S\setminus \Bbbk [Q]$. Then we can write $a=\alpha_0 X_0+\alpha_1X_2+\cdots + \alpha_{n-1}X_n+m$, where $m$ is a nonzero element in $\Bbbk$. Clearly, $a^k\neq 0$, for any $k\geq 1$. This proves that any nilpotent element in $S$ must be in $\Bbbk[Q]$. 
    \end{proof}

    We conclude this section with a theorem related to the nil clean property of quandle rings.
    \begin{thm}
        Let $Q$ be a trivial quandle and $S$ be the extended quandle ring of $\Bbbk[Q].$ Then neither $\Bbbk[Q]$  nor $S$ is nil clean.
    \end{thm}
    \begin{proof}
        For a trivial quandle $Q,$ it is known that any idempotent of $\Bbbk[Q]$ is of the form $x_0+\delta,$ where $x_0\in Q$ is a fixed element and $\delta\in\Delta_\Bbbk(Q)$ is arbitrary \cite[Proposition 4.1]{BPS-2022}. Therefore the sum of an idempotent and nilpotent element is of the form $s=x_0+\delta_1+\delta_2.$ As $\Delta_\Bbbk(Q)$ is an ideal, so $s=x_0+\delta_3,$ where $\delta_3=\delta_1+\delta_2\in\Delta_\Bbbk(Q).$ It follows that the sum of an idempotent and a nilpotent in $\Bbbk[Q]$ is always an idempotent and hence $\Bbbk[Q]$ is not a nil-clean ring.

        Let $a$ be a nilpotent element in $S$. Then by Proposition~\ref{prop 5.5} and Theorem~\ref{thm 5.4}, $a\in\Delta_\Bbbk(Q).$ The sum of an element in $\Delta_\Bbbk(Q)$ and an element in $S\setminus \Bbbk[Q]$ is clearly in $S\setminus \Bbbk[Q]$, and therefore an element in $\Bbbk[Q]$ cannot be written as a sum of a nilpotent and an idempotent (in $S\setminus \Bbbk[Q]$, if any). Clearly, every element in $\Bbbk[Q]$ is not an idempotent. Thus there are elements in $\Bbbk[Q]$,  and therefore in $S$, that cannot be written as a sum of a nilpotent and an idempotent. Thus, $S$ cannot be a nil-clean ring.
    \end{proof}

We now focus on the extended quandle ring of the Joyce quandle where the ground ring is an integral domain. The multiplication table of the Joyce quandle of order 3 is given by

\[
\begin{tabular}{c | c c c}
$\ast$ & $e_0$ & $e_1$ & $e_2$  \\
\cline{1-4}
$e_0$& $e_0$ & $e_o$ & $e_0$ \\
$e_1$& $e_2$ & $e_1$ & $e_1$ \\
$e_2$& $e_1$ & $e_2$ & $e_2$ \\
\end{tabular}
\]

Let $Q$ be the Joyce quandle of order 3 and $\Bbbk$ be an integral domain. For the rest of this section, we study units of the extended quandle ring of $\Bbbk[Q]$. 

We point out to the reader that there are no units in the regular quandle ring. Let $a=\alpha_0e_0+\alpha_1e_1+\alpha_2e_2+\gamma_1$ be a unit. Then there exists a nonzero element $b=\beta_0e_0+\beta_1e_1+\beta_2e_2+\gamma_1$ such that 
$$(\alpha_0e_0+\alpha_1e_1+\alpha_2e_2+\gamma_1)(\beta_0e_0+\beta_1e_1+\beta_2e_2+\gamma_2)=e$$
and 
$$(\beta_0e_0+\beta_1e_1+\beta_2e_2+\gamma_2)(\alpha_0e_0+\alpha_1e_1+\alpha_2e_2+\gamma_1)=e.$$

From the first equation, we get the following system of equations:
\begin{equation}\label{EQ6.4}
\begin{cases} 
\alpha_0\beta_0+\alpha_0\beta_1+\alpha_0\beta_2+\alpha_0\gamma_2+\gamma_1\beta_0=0,\\
\alpha_1\beta_1+\alpha_1\beta_2+\alpha_1\gamma_2+\alpha_2\beta_0+\gamma_1\beta_1=0,\\
\alpha_1\beta_0+\alpha_2\beta_1+\alpha_2\beta_2+\alpha_2\gamma_2+\gamma_1\beta_2=0,\\
\gamma_1\gamma_2=e.
\end{cases}
\end{equation} 

From the second equation, we get the following system of equations:
\begin{equation}\label{EQ6.5}
\begin{cases} 
\beta_0\alpha_0+\beta_0\alpha_1+\beta_0\alpha_2+\beta_0\gamma_1+\gamma_2\alpha_0=0,\\
\beta_1\alpha_1+\beta_1\alpha_2+\beta_1\gamma_1+\beta_2\alpha_0+\gamma_2\alpha_1=0,\\
\beta_1\alpha_0+\beta_2\alpha_1+\beta_2\alpha_2+\beta_2\gamma_1+\gamma_2\alpha_2=0,\\
\gamma_1\gamma_2=e.
\end{cases}
\end{equation} 

Since $a$ and $b$ are nonzero, not all $\alpha_i$'s and $\beta_j$'s are zero. Then $a\in S$ is a unit if and only if the equations \eqref{EQ6.4} and \eqref{EQ6.4} are both solvable. 

Let $\Bbbk=\mathbb{F}_2$. We find the units in the extended quandle ring of $\mathbb{F}_2[Q]$, where $Q$ is the Joyce quandle. 

From equations \eqref{EQ6.4} and \eqref{EQ6.5}, we get $\gamma_1=\gamma_2=1$, and 
\begin{equation}\label{EQ6.6}
\begin{cases}
\alpha_0(\beta_1+\beta_2)=\beta_0(\alpha_1+\alpha_2),\\
(\alpha_1-\alpha_0)\beta_2=\alpha_2(\beta_1-\beta_0),\\
(\alpha_2-\alpha_0)\beta_1=(\beta_2-\beta_0)\alpha_1.\\
\end{cases}
\end{equation} 

If $\alpha_0=0$, we have $\beta_0=0$ or $\alpha_1=\alpha_2$. Assume that $\beta_0=0$. then from the second and third equations, we have $\alpha_1\beta_2=\alpha_2\beta_1$. If $\alpha_1\beta_2=\alpha_2\beta_1=1$, then it is easy to see that $(e_1+e_2+1)$ is a unit, which is in fact a self-inverse. If $\alpha_1\beta_2=\alpha_2\beta_1=0$, then equation 2 in \eqref{EQ6.4} gives $\alpha_1(1+\beta_1)+\beta_1=0$. if $\beta_1=1$, then we also get $\beta_1=0$, which is a contradiction. If $\beta_1=0$, we have  $\beta_1=0$. Using $\alpha_0=\beta_0=\beta_1=\beta_1=0$ in the third equation in \eqref{EQ6.4} we get $\alpha_2(1+\beta_2)+\beta_2=0$. If $\beta_2=1$, we also get $\beta_2=0$, which is a contradiction. If $\beta_2=0$, then $\alpha_2=0$. This means the only unit in this case is 1.  

If $\alpha_1=\alpha_2$, then the last two equations in \eqref{EQ6.6} gives us

$$(\alpha_1-\alpha_0)\beta_2=\alpha_1(\beta_1-\beta_0),$$
$$(\alpha_1-\alpha_0)\beta_1=\alpha_1(\beta_1-\beta_0),$$

which leads to $\alpha_1=0$ or $\beta_0+\beta_1+\beta_2=0$. Clearly, $\alpha_1\neq 0$ because otherwise $\alpha_0=\alpha_1=\alpha_2=0$, which is a contradiction. 

The equation $\beta_0+\beta_1+\beta_2=0$ leads to the triples $(0,0,0),(1,1,0),(1,0,1),(0,1,1)$, which correspond to the elements $0$, $e_0+e_1+1$, $e_0+e_2+1$, and $e_1+e_2+1$, respectively. Clearly,  
$$(e_0+e_1+1)(e_0+e_2+1)=1=(e_0+e_2+1)(e_0+e_1+1)$$
and 
$$(e_1+e_2+1)(e_1+e_2+1)=1.$$
Now let $\alpha_0=1$. Then from the first equation in \eqref{EQ6.4}, we get $\beta_1+\beta_2=1$. Substituting $1$ for $\beta_1+\beta_2$ in \eqref{EQ6.6}, we get $\beta_0(\alpha_1+\alpha_2)=1$. This implies $\beta_0=1$ and $\alpha_1+\alpha_2=1$, which correspond to the elements $e_0+e_1+1$ and $e_0+e_2+1$. 
We have shown that there are only four units in the extended quandle ring $S$, where the quandle $Q$ is the Joyce quandle of order 3 and the ground ring is the finite field $\mathbb{F}_2$. They are $1$,$e_0+e_1+1$, $e_0+e_2+1$, and $e_1+e_2+1$.      

\section{Zero-divisor graphs of quandle rings}\label{S6}

In this section we present some results on zero-divisor graphs of extended quandle rings.  A wealth of theory on zero-divisor graphs can be found in \cite{AABC-2021}, but we present the necessary background here.    

A graph $\Gamma = (V, E)$ is an ordered pair of two sets, where $V$ is called the {\it vertex set} (whose elements are called the {\it vertices} of the graph) and $E$ is called the {\it edge set} (whose elements are called the {\it edges} of the graph). $\Gamma = (V, E)$ is a {\it simple} graph if $E \subseteq \binom{V}{2}$, that is, each $e \in E$ is of the form $e = \{u,v\}$ for some $u \neq v \in V$, and we say $u$ and $v$ are {\it adjacent}. For the most part we will focus on non-simple graphs here.  An edge $e \in E$ of the form $e = \{u\}$ is called a {\it loop} at the vertex $u$ (so a loop ``joins a vertex to itself").  A {\it directed graph} (or {\it digraph}) $\Gamma = (V, E)$ (sometimes denoted $\vec{\Gamma} = (V, \vec{E})$) is a graph where edges are ordered pairs of vertices, and if $e = (u,v) \in E$, we say there is an edge from $u$ to $v$; we also say the the edge goes out from $u$ and in to $v$. A directed loop is of the form $e = (u, u)$.  

For a simple graph $\Gamma = (V, E)$ and for $u \in V$, we define the {\it degree} of $u$ by $\deg(u) = \#\{e \in E : u \in e\}$.  For a directed graph $\Gamma = (V, E)$ and for $u \in V$, we define the {\it out-degree} of $u$ by $\deg^+(u) = \#\{v \in V : (u,v) \in E\}$, that is, the number of edges which start at $u$. Similarly, we define the {\it in-degree} of $u$ by $\deg^-(u) = \#\{v \in V : (v,u) \in E\}$, that is, the number of edges which terminate at $u$. Note that a directed loop $(u,u) \in E$ is counted in the out-degree and in-degree of $u$. Note also that 

\begin{equation}
\sum_{u \in V}\deg^+(u) = \sum_{u \in V}\deg^-(u) = |E|
\end{equation}

\noindent but this tells us little information about the directed graph.  

Let $R$ denotes a commutative ring.  The {\it zero-divisor graph} of $R$, denoted $\Gamma(R)$, is the graph with vertex set $Z(R)$, the set of zero-divisors of $R$ and two vertices $u, v$ are adjacent if $uv = 0$.  If $R$ is a non-commutative ring, then the zero-divisor graph, sometimes denoted $\vec{\Gamma}(R)$, is a directed graph with vertex set $ Z(R)$ and $(u,v)$ is a directed edge from $u$ to $v$ if $uv = 0$. Note that in a zero-divisor graph, a loop at $u$ corresponds to $u^2 = 0,$ which implies that $u$ nilpotent of index $2.$
%\vskip 0.1 in
\noindent Unlike group rings when the group is trivial, quandle rings when the quandle is trivial seems to be an area that requires a thorough study as evident in this section. 

\subsection{Zero-divisor graph of the extended quandle ring of a trivial quandle}

Let us consider the zero-divisor graph of the extended quandle ring $S$ of $\mathbb{F}_2[Q],$ where $Q=\{x_0,x_1\}$ is the trivial quandle of order $2.$ Then
$$S = \{0, 1, x_0,x_1, x_0+1,x_1+1,x_0+x_1, x_0+x_1+1\}$$ and its zero divisor graph  is shown below.

%\begin{minipage}{0.58\textwidth}
\begin{center}
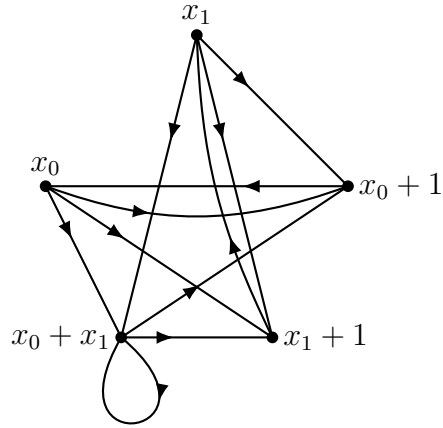


\begin{tikzpicture}[
    >=Latex,
    edge/.style={
    draw,
    thick,
    postaction={decorate},
    decoration={
        markings,
        mark=at position 0.35 with {\arrow{Latex}}
    }
},
]

% Vertices

\coordinate (A) at (0,0);
\coordinate (B) at (2,2);
\coordinate (C) at (4,0);
\coordinate (D) at (1,-2);
\coordinate (E) at (3,-2);

% Solid points
\fill (A) circle (2.2pt);
\fill (B) circle (2.2pt);
\fill (C) circle (2.2pt);
\fill (D) circle (2.2pt);
\fill (E) circle (2.2pt);

% Labels
\node[above] at (A) {$x_0$};
\node[above] at (B) {$x_1$};
\node[right] at (C) {$x_0+1$};
\node[left] at (D) {$x_0+x_1$};
\node[right] at (E) {$x_1+1$};

% Directed edges
\draw[edge] (D) to[out=-40,in=-120, looseness=10,out distance=2cm,in distance=2cm]  (D);
\draw[edge,bend right=20]  (A) to (C);
\draw[edge] (C) to (A);
\draw[edge] (A) to (D);
\draw[edge] (A) to (E);
\draw[edge] (B) to (D);
\draw[edge] (B) to (E);
\draw[edge,bend left=12] (E) to (B);
\draw[edge] (D) to (C);
\draw[edge] (D) to (E);
\draw[edge] (B) to (C);
\end{tikzpicture}
\captionof{figure}{Zero-divisor graph of the extended quandle ring of the trivial quandle of order 2}
\label{fig1}
\end{center}
%\end{minipage}%
\hfill % Adds space between them
%\begin{minipage}{0.38\textwidth}
\begin{center}
    \begin{tabular}{|c|c|c|}
        \hline
        Vertex & In-degree & Out-degree \\
        \hline
        $x_0$ & 1 & $2^2-1$ \\
        $x_1$  & 1 & $2^2-1$ \\
         $x_0+1$  & $2^2-1$ & 1 \\
          $x_1+1$  & $2^2-1$ & 1 \\
           $x_0+x_1$  & $2^2-1$ & $2^2-1$ \\
        \hline
    \end{tabular}
    \end{center}
\captionof{table}{In-degree and out-degree table of Figure~\ref{fig1}}
\label{tab1}
%\end{minipage} 

\vskip 0.2in

We now compute the zero divisors of the extended quandle ring of a trivial quandle whose ground ring is $\mathbb{F}_2$. 

%From the indegree-outdegree tables of the graphs associated with the extended trivial quandle ring and extended Joyce quandle ring over $\mathbb{F}_2$ (Table~\ref{tab2},Table~\ref{tab3}), a remarkable mirror symmetry was observed. More precisely, whenever there exists a vertex $x$ (so $x$ is a zero divisor) with in-degree $i$ and out-degree $o$, there also exists a vertex (zero divisor) with in-degree $o$ and out-degree $i$ and the sum of these two corresponding elements in the quandle ring is never a zero divisor. This observation motivated us to study the mirror symmetry property theoretically. In fact, we prove that this mirror symmetry holds for the graph associated with the extended quandle ring of $\Bbbk[Q]$ whenever $Q$ is a finite trivial quandle and $\Bbbk=\mathbb{F}_2.$

\vskip 0.1in

%%%%%%%%
    It is known that in a finite ring with unity, every nonzero element is either a zero divisor or a unit but never both. The following proposition follows from Example~\ref{exmp 7.3}.
    \begin{prop}
         Let $\Bbbk=\mathbb{F}_2$ and $Q$ be a trivial quandle of finite order. Then the set of zero divisors $Z$ of the extended quandle ring of $\Bbbk[Q]$ is $$Z=\{u\in S\setminus\{0\}\mid u\neq \delta+1\},$$ where $\delta\in\Delta_\Bbbk(Q).$
    \end{prop}

    \begin{rmk}\label{rmk6.2}
     If $ u=\alpha x+\delta+\gamma\in Z,$ then we must have one of the following:
    \begin{enumerate}
        \item 
        $\alpha=0,\,\gamma=0,\,(u=\delta\setminus\{0\}).$
        \item 
        $\alpha=1,\,\gamma=0\,(u=x+\delta).$
        \item 
        $\alpha=1,\,\gamma=1\,(u=x+\delta+1).$
    \end{enumerate}
  
    In particular, for the trivial quandle $Q_3=\{x_0,\,x_1\,x_2\}$ of order $3,$ $$Z=\{x_0+x_1,\,x_0+x_2,\,x_1+x_2,\,x_0,\,x_1\,x_2,\,x_0+x_1+x_2,\,x_0+1,\,x_1+1,\,x_2+1,\,x_0+x_1+x_2+1\}.$$
    \end{rmk} 

More generally, we have the following theorem, which follows from Theorem~\ref{thm 5.2}.
\begin{thm}\label{thm 7.5}
    Let $\mathbb{F}$ be a finite field with unity $e,$ $Q$ be a finite order trivial quandle and $S$ be the extended quandle ring of $\mathbb{F}[Q].$ An element $u=\alpha x+\delta+\gamma\in S,$ where $\alpha,\gamma\in \mathbb{F},\,x\in Q,\,\delta\in\Delta_\mathbb{F}(Q)$ is a zero divisor if and only if the following system of equations is inconsistent over $\mathbb{F}$ for any choices of $(\alpha_1,\,\delta_1,\,\gamma_1)\in(\mathbb{F},\,\Delta_\mathbb{F}(Q),\,\mathbb{F}).$
    \begin{equation}  \label{eqn7.4}
    \begin{array}{rcccl}
    \alpha\alpha_1+\gamma_1\alpha+\gamma\alpha_1&=& 0 & &\\
    \alpha_1\delta+\gamma_1\delta+\gamma\delta_1&=& 0 & &\\
    %&=&\alpha\delta_1+\gamma_1\delta+\delta_1\gamma\\
    \alpha_1\delta &=&\alpha\delta_1 & &\\
    \gamma\gamma_1&=& e .& &
    \end{array}
\end{equation}
\end{thm}

Now we present the zero-divisor graph of the extended quandle ring of the trivial quandle of order 3 and the table of the in-degrees and out-degrees of its vertices. 

\begin{center}
\begin{tikzpicture}[
    >=Latex,
    edge/.style={
    draw,
    thick,
    postaction={decorate},
    decoration={
        markings,
        mark=at position 0.35 with {\arrow{Latex}}
    }
},
]

% Vertices
\coordinate (A) at (-5,-1);
\coordinate (B) at (-6,1);
\coordinate (C) at (-6,4);
\coordinate (D) at (-4,7);
\coordinate (E) at (-1.5,9);
\coordinate (F) at (1,8);
\coordinate (G) at (3,6.5);
\coordinate (H) at (5,4);
\coordinate (I) at (5,2);
\coordinate (J) at (4,0);
\coordinate (K) at (-1,-1);
%\coordinate (L) at (2,-1);
%\coordinate (M) at (-1,-3);
%\coordinate (N) at (-3,-2);
%\coordinate (O) at (-4.5,-1);

% Solid points
\fill (A) circle (2.2pt);
\fill (B) circle (2.2pt);
\fill (C) circle (2.2pt);
\fill (D) circle (2.2pt);
\fill (E) circle (2.2pt);
\fill (F) circle (2.2pt);
\fill (G) circle (2.2pt);
\fill (H) circle (2.2pt);
\fill (I) circle (2.2pt);
\fill (J) circle (2.2pt);
\fill (K) circle (2.2pt);
%\fill (L) circle (2.2pt);
%\fill (M) circle (2.2pt);
%\fill (N) circle (2.2pt);
%\fill (O) circle (2.2pt);

% Labels
\node[left] at (A) {$x_0$};
\node[left] at (B) {$x_1$};
\node[left] at (C) {$x_2$};
\node[left] at (D) {$x_0+x_1$};
\node[above left] at (E) {$x_0+x_2$};
\node[right] at (F) {$x_1+x_2$};
\node[above right] at (G) {$x_0+x_1+x_2$};
\node[above right] at (H) {$x_0+1$};
\node[right] at (I) {$x_1+1$};
\node[right] at (J) {$x_2+1$};
\node[below] at (K) {$x_0+x_1+x_2+1$};
%\node[right] at (L) {$x_1+x_2+1$};
%\node[below] at (M) {$1$};
%\node[left] at (N) {$x_0+x_1+x_2$};
%\node[left] at (O) {$x_0+x_1+x_2+1$};

\draw[edge] (A) -- (K);
\draw[edge] (H) to[bend right=5] (A);
\draw[edge] (A) -- (J);
\draw[edge] (A) -- (I);
\draw[edge] (A) -- (F);
\draw[edge] (A) -- (H);
\draw[edge] (A) -- (E);
\draw[edge] (A) -- (D);
% from and to x_1
\draw[edge] (B) -- (K);
\draw[edge] (B) -- (J);
\draw[edge] (B) to[bend right=5] (I);
%\draw[edge] (B) to[bend right=35] (I);
\draw[edge] (B) -- (H);
\draw[edge] (I) -- (B);
\draw[edge] (B) -- (F);
\draw[edge] (B) -- (E);
\draw[edge] (B) -- (D);
% from and to x_2
\draw[edge] (C) -- (K);
\draw[edge] (C) to[bend left=5]  (J);
%\draw[edge] (J) to[bend left=35] (C);
\draw[edge] (J) -- (C);
\draw[edge] (C) -- (I);
\draw[edge] (C) -- (H);
%\draw[edge] (C) to[bend right=45] (H);
\draw[edge] (C) -- (F);
\draw[edge] (C) -- (E);
\draw[edge] (C) -- (D);
% to x_0+1
\draw[edge] (D) to[out=60,in=120, looseness=10,out distance=2cm,in distance=3cm]  (D);

\draw[edge] (D) -- (K);
\draw[edge] (D) -- (J);
\draw[edge] (D) -- (I);
\draw[edge] (D) -- (H);
\draw[edge] (F) -- (D);
\draw[edge] (D) to[bend left=25] (F);
\draw[edge] (D) to[bend left=25] (E);
\draw[edge] (E) -- (D);
% from and to x_0++x_2
\draw[edge] (E) to[out=30,in=120, looseness=10,out distance=2cm,in distance=3cm]  (E);
\draw[edge] (E) -- (K);
\draw[edge] (E) -- (J);
\draw[edge] (E) -- (I);
\draw[edge] (E) -- (H);
\draw[edge] (G) -- (E);
\draw[edge] (E) to[bend left=25] (F);
\draw[edge] (F) -- (E);
% from and to x_2+x_1
\draw[edge] (F) to[out=40,in=120, looseness=10,out distance=2cm,in distance=3cm]  (F);
\draw[edge] (F) -- (K);
\draw[edge] (F) -- (J);
\draw[edge] (F) -- (I);
\draw[edge] (F) -- (H);
\draw[edge] (G) -- (F);
%\draw[edge] (F) -- (D);
% from and to x_0+x_1+x_2
\draw[edge] (G) -- (H);
\draw[edge] (G) to[bend left=5] (I);
\draw[edge] (G) -- (J);
\draw[edge] (G) -- (K);
\draw[edge] (G) -- (D);
%\draw[edge] (K) to[bend right=35] (G);
\draw[edge] (K) to[bend right=15] (G);
\end{tikzpicture}
\captionof{figure}{Zero-divisor graph of the extended quandle ring of the trivial quandle of order 3}
\label{fig3}
\end{center}

\begin{table}[H]
    \centering
    %\caption{Sample Table Description}
    \label{tab:sample_table3}
    \begin{tabular}{|c|c|c|}
        \hline
        Vertex & In-degree & Out-degree \\
        \hline
        $x_0$ & 1 & $2^3-1$ \\
        $x_1$  & 1 & $2^3-1$ \\
         $x_2$  & 1 & $2^3-1$ \\
          $x_0+1$  & $2^3-1$ & 1 \\
           $x_1+1$  & $2^3-1$ & 1 \\
           $x_2+1$ & $2^3-1$ & 1 \\
        $x_0+x_1$  & $2^3-1$ & $2^3-1$ \\
         $x_0+x_2$  & $2^3-1$ & $2^3-1$ \\
          $x_1+x_2$  & $2^3-1$ & $2^3-1$ \\
         $x_0+x_1+x_2$  & 1 & $2^3-1$ \\
          $x_0+x_1+x_2+1$  & $2^3-1$ & 1 \\
           
        \hline
  \end{tabular}
  \captionof{table}{In-degree and out-degree table of Figure~\ref{fig3}}
\label{tab3}
\end{table}

From the indegree-outdegree tables of the zero-divisor graphs associated with the extended quandle ring of the trivial quandle over $\mathbb{F}_2$ (Table~\ref{tab1}, Table~\ref{tab3}), a remarkable mirror symmetry is observed. More precisely, for each $u \in Z(R)$ (i.e., a vertex of the zero-divisor graph),  there is $v \in Z(R)$ such that $\deg^+(u) = \deg^-(v)$ and $\deg^-(u) = \deg^+(v)$.  In this case, we say $u$ and $v$ exhibit \textit{mirror symmetry}. In fact, we also observe that for every $m, \ell \in \mathbb{N}$,

\begin{footnotesize}
   $$ \# \{u \in Z(R) : \deg^+(u) = m \ \text{and} \ \deg^-(u) = \ell \} \ = \# \{v \in Z(R) : \deg^+(v) = \ell \ \text{and} \ \deg^-(v) = m \} $$
\end{footnotesize}

Additionally, if we take any two vertices that exhibit mirror symmetry, then the sum of these two elements in the quandle ring is never a zero divisor.  These observations motivated us to study these mirror symmetry properties theoretically; in fact, we prove that the above mirror symmetry properties hold for the extended quandle ring of $\mathbb{F}_2[Q]$ whenever $Q$ is a finite trivial quandle. 

\vskip 0.1in 

Now we compute the in-degrees and out-degrees of the vertices of the zero-divisor graph of the extended quandle ring of the trivial quandle. From Remark~\ref{rmk6.2}, we know that any zero divisor of the extended quandle ring $S$ of $\mathbb{F}_2[Q],$ where $Q$ is a trivial quandle of finite order, (say $n>1$) must be of the form $\delta,\,x+\delta$ or $x+\delta+1$ for some $x\in Q$ and $\delta\in \Delta_{\mathbb{F}_2} Q=\Delta.$ Let $z$ be a zero divisor of $S.$
\vskip 0.1in
\textbf{Case 1.} $z=\delta$ for some $\delta\in\Delta.$ As $z\in Z,$ therefore there exists $y=\alpha x+\delta_1+\gamma\in Z,\,y\neq 0$ such that $zy=0.$ Number of such $y'$s will give us the out degree of $z$ in the zero divisor graph of $S.$ From the equation $zy=0,$ we get that $\alpha\delta x+\delta\delta_1+\gamma\delta=0.$ Since $Q$  is trivial, $\delta\delta_1=0$ and $\delta x=\delta.$ Therefore, it follows that $\alpha\delta+\gamma\delta=0.$ Comparing the coefficients of $\delta$ on both sides, we get $\alpha=\gamma.$ So $y$ must be of the form $x+\delta_1+1$ or $\delta_1,$ where $\delta_1\in\Delta$ is arbitrary. 

\vskip 0.1in 

It is known that for a finite field $\mathbb{F}$ with unity, the cardinality of $\Delta_\Bbbk(Q)$ is equal to $f^q,$ where $f$ is the order of $\mathbb{F},$ and $q$ is the dimension of $\Delta_\Bbbk(Q),$ which is $n-1.$ Consequently, the cardinality of $\Delta$ is $2^{n-1}.$ Therefore the out-degree of $z$ is $2^{n-1}+2^{n-1}-1=2^n-1.$ In particular, when $n=3,$ out-degree of all the elements $\Delta\setminus\{0\}=\{x_0+x_1,\,x_0+x_2,\,x_1+x_2\}$ is $7.$

\vskip 0.1in 

Next, we consider in-degree of $z.$ From the equation $yz=0,$ we get $\alpha x\delta+\gamma\delta=0.$ Clearly, $x\delta =0,$ so we have $\gamma=0$ (since $\delta\neq 0).$ Therefore $y$ is of the form $\alpha x+\delta_1.$ So $y=x+\delta_1$ or $y=\delta_1.$ It follows that the in-degree of $z$ is $2^n-1.$ In particular, when $n=3,$ in-degree of all the elements $\Delta\setminus\{0\}=\{x_0+x_1,\,x_0+x_1,\,x_1+x_2\}$ is $7.$
\vskip 0.1 in
\textbf{Case 2.} $z=x+\delta\in Z.$ We apply similar technique as that of Case 1. 

For out-degree computation, we have the following equation $$\alpha x+\alpha \delta+\gamma x+\gamma \delta=0,$$ which gives us $\alpha=\gamma.$ Therefore $y=x+\delta_1+1$ or $y=\delta_1.$ Out-degree $ =2^n-1.$ In particular, when $n=3,$ out-degree of all the elements of the form $\{x+\delta\}=\{x_0,\,x_1\,x_2,\,x_0+x_1+x_2\}$ is $7.$

\vskip 0.1in 

For in-degree computation, we have the following equation $$\alpha x+\delta_1+\gamma x+\gamma\delta=0.$$ As $y$ is also a zero divisor, therefore we have following $3$ choices.
\begin{enumerate}
    \item 
    $\alpha=0,\,\gamma=0.$ Then we have $\delta_1=0,$ which implies $y=0,$  a contradiction.
    \item 
    $\alpha=1,\,\gamma=0.$ Then we have $x+\delta_1=0,$ which is not possible as the augmentation value of $x+\delta_1$ is $1\neq 0.$
    \item 
    $\alpha=1,\,\gamma=1.$ Then we have $\delta=\delta_1,$ which leads to $y=x+\delta+1.$ (only one choice)
\end{enumerate}
In-degree= 1. In particular, when $n=3,$ in-degree of all the elements of the form $\{x+\delta\}=\{x_0,\,x_1,\,x_2,\,x_0+x_1+x_2\}$ is $1.$ 
\vskip 0.1in
\textbf{Case 3.} $z=x+\delta+1\in Z.$ We apply similar technique as that of Case 1. 

For out-degree computation, we have the following equation $$\gamma x+\alpha\delta+\gamma\delta+\delta_1+\gamma=0.$$ Clearly, $\gamma$ (the constant term) must be $0,$ which leaves us with $\alpha\delta+\delta_1=0.$ When $\alpha=0,$ we have $\delta_1=0,$ and consequently $y=0,$ which is not possible. Therefore $\alpha=1,$ and hence $y=x+\delta_1=x+\delta.$ Out-degree $=1.$ 

\vskip 0.1in

In particular, when $n=3,$ out-degree of all the elements of the form $\{x+\delta+1\}=\{x_0+1,\,x_1+1,\,x_2+1,\,x_0+x_1+x_2+1\}$ is $1.$ 

\vskip 0.1in

For in-degree computation, we have the following equation
\begin{equation}\label{eqn 7.5}
    \gamma x+\gamma\delta+\gamma=0.
\end{equation} 
Comparing the terms belonging to $\mathbb{F}_2,$ we get $\gamma=0.$ For $\gamma=0,$ equation~\ref{eqn 7.5} is always satisfied. Therefore $y=\alpha x+\delta_1,$ where $\alpha\in \mathbb{F}_2,\,\delta_1\in \Delta.$ $y$ is of the form $x+\delta_1$ or $\delta_1$ and hence in-degree =$2.2^{n-1}-1=2^n-1.$ 

\vskip 0.1in

In particular, when $n=3,$ in-degree of all the elements of the form $\{x+\delta+1\}=\{x_0+1,\,x_1+1,\,x_2+1,\,x_0+x_1+x_2+1\}$ is $7.$ 

\vskip 0.1in 

The above findings can be summarized as follows.
\begin{prop}
Let $S$ be the extended quandle ring of $\mathbb{F}_2[Q],$ where $Q$ is trivial quandle of order $n>1.$ Let $\delta$ be an arbitrary element of $\in\Delta_{\mathbb{F}_2}[Q]\setminus\{0\}$ and $x_0\in Q.$ Then we have the following:

\begin{table}[H]
\centering
\begin{tabular}{|c|c|c|}
\hline
\text{vertex} & \text{in-degree} & \text{out-degree} \\
\hline
$\delta$ & $2^n-1$  & $2^n-1$ \\
\hline
$x_0+\delta$ & $1$ & $2^n-1$ \\
\hline
$x_0+\delta+1$ & $2^n-1$ & $1$ \\
\hline
\end{tabular}
\caption{In-degrees and out-degrees of the zero divisor graph of the extended quandle ring of the trivial quandle of order $n$ over $\mathbb{F}_2.$}
\label{tab 1}
\end{table}

\FloatBarrier

\end{prop}
From Table~\ref{tab 1}, we can see that any two elements of the form $x_0+\delta$ and $x_0+\delta_1+1$ exhibit mirror symmetry with respect to their in-degrees and out-degrees. If we add them, it gives us  the unity $1$ which is clearly not a zero divisor. For completeness, we can also say that $\delta$ is mirror symmetric to itself, and $\delta+\delta=0,$ which is not a zero divisor. Thus, we have the following theorem.
\begin{thm}
    Let $Z$ be the set of zero divisors of the extended quandle ring $S$ of $\mathbb{F}_2[Q],$ where $Q$ is a finite order trivial quandle. Then the sum of any pair of elements $(z_1,\,z_2)\in Z\times Z$ exhibiting mirror symmetry is not a zero divisor.
\end{thm}

\vskip 0.1in

\subsection{Zero-divisor graph of the extended quandle ring of the Joyce quandle}

We now turn our focus onto the zero-divisors of the extended quandle ring of the Joyce quandle over an integral domain. We will also demonstrate that the vertices of its zero-divisor graph also exhibit the mirror symmetry. 

\vskip 0.1in

We first find the out-degrees of the vertices that correspond to elements in the regular quandle ring. The out-degree of a nonzero element $a\in S$ is equal to the number of nonzero $b$'s such that $ab=0$. 

\vskip 0.1in

Consider the nonzero elements $\alpha_0e_0+\alpha_1e_1+\alpha_2e_2\in \Bbbk[Q]$ and $\beta_0e_0+\beta_1e_1+\beta_2e_2+\gamma\in S$, and set 
$$(\alpha_0e_0+\alpha_1e_1+\alpha_2e_2)\cdot (\beta_0e_0+\beta_1e_1+\beta_2e_2+\gamma)=0.$$

It is easy to see that the element $\alpha_0e_0+\alpha_1e_1+\alpha_2e_2$ is a left-zero divisor if and only if the following system of equations is solvable in $\Bbbk$ such that not all $\beta_j$'s are zero:

$$\alpha_0(\beta_0+\beta_1+\beta_2+\gamma)=0,$$
$$\alpha_1(\beta_1+\beta_2+\gamma)+\alpha_2\beta_0=0,$$
$$\alpha_2(\beta_1+\beta_2+\gamma)+\alpha_1\beta_0=0.$$
\vskip 0.1in
The first equation implies that $\alpha_0=0$ or $\beta_0+\beta_1+\beta_2+\gamma=0$. We let $\Bbbk=\mathbb{F}_2$, and consider two cases: $\alpha_1=\alpha_2$, $\alpha_1\neq \alpha_2$. \vskip 0.1in

\textbf{Case 1.} $\alpha_1=\alpha_2$. In this case, the left zero-divisor $e_1+e_2$ has out-degree $2^3-1=7$ ($\alpha_0=0$, $\alpha_1=\alpha_2=1$ and $\gamma=\beta_0+\beta_1+\beta_2$). Also, the out-degree of the vertices $e_0$ and $e_0+e_1+e_2$ is $2^3-1=7$ ($\alpha_0=\alpha_1=\alpha_2=1$ and $\gamma=\beta_0+\beta_1+\beta_2$). \vskip 0.1in

\textbf{Case 2.} $\alpha_1\neq \alpha_2$. We have the following:
\begin{itemize}
\item The out-degree of $e_1$ and $e_2$ is $2^2-1=3$ ($\alpha_0=0$,  $\alpha_1\neq \alpha_2$ and $\gamma=\beta_1+\beta_2$). 
\item The out-degree of $e_0+e_1$ and $e_0+e_2$ is $2^2-1=3$ ($\alpha_0=1$,  $\alpha_1\neq \alpha_2$ and $\gamma=\beta_1+\beta_2$). 
\end{itemize}

We now compute the in-degrees of the zero-divisors in the quandle ring $\mathbb{F}_2[Q]$, where $Q$ is the Joyce quandle. \vskip 0.1in

Consider an arbitrary element $\beta_0e_0+\beta_1e_1+\beta_2e_2+\gamma$ in the extended quandle ring. Then 
$(\beta_0e_0+\beta_1e_1+\beta_2e_2+\gamma)\cdot e_j=0$ implies $\beta_j=\gamma$ and $\beta_k=0$ for $j\neq k$. Thus the in-degree of the elements $e_j$ is 1. Also, $(\beta_0e_0+\beta_1e_1+\beta_2e_2+\gamma)\cdot (e_0+e_1+e_2)=0$ implies $\beta_0=\beta_1=\beta_2=\gamma$, and therefore the in-degree of $e_0+e_1+e_2$ is also 1. 
\vskip 0.1in
Consider $(\beta_0e_0+\beta_1e_1+\beta_2e_2+\gamma)\cdot (e_0+e_j)=0$, where $j\neq 0.$ Then we have $\gamma=0$ and $\alpha_1=\alpha_2$. Thus the in-degree for $(e_0+e_1)$ and $(e_0+e_2)$ is 3 as there are $2^2-1=3$ nonzero choices for the triple $(\alpha_0, \alpha_1, \alpha_1)$.
\vskip 0.1in
Set $(\beta_0e_0+\beta_1e_1+\beta_2e_2+\gamma)\cdot (e_1+e_2)=0$. Then we obtain $\gamma = 0$. This means we have $2^3-1=7$ nonzero choices for the triple $(\alpha_0, \alpha_1, \alpha_2)$. 
\vskip 0.1in
Now we compute the in-degrees and out-degrees of the elements in the extended quandle ring but not in the regular quandle ring. 

Consider the nonzero elements $\alpha_0e_0+\alpha_1e_1+\alpha_2e_2+\gamma \in S\setminus \Bbbk[Q]$ and $\beta_0e_0+\beta_1e_1+\beta_2e_2\in \Bbbk[Q]$, and set 
$$(\alpha_0e_0+\alpha_1e_1+\alpha_2e_2+\gamma)\cdot (\beta_0e_0+\beta_1e_1+\beta_2e_2)=0.$$

It is easy to see that the element $\alpha_0e_0+\alpha_1e_1+\alpha_2e_2+\gamma$ is a left-zero divisor if and only if the following system of equations is solvable in $\Bbbk$ such that $\beta_0\beta_1\beta_2\neq 0$:
\begin{equation}\label{EQ6.7}
\begin{cases} 
\alpha_0(\beta_0+\beta_1+\beta_2)+\gamma\beta_0=0,\\
\alpha_1(\beta_1+\beta_2)+\alpha_2\beta_0+\gamma\beta_1=0,\\
\alpha_1\beta_0+\alpha_2(\beta_1+\beta_2)+\gamma\beta_2=0.\\
\end{cases} 
\end{equation} 

Adding the last two equations gives us 
$$(\beta_1+\beta_2)(\alpha_1+\alpha_2+\gamma)+(\alpha_1+\alpha_2)\beta_0=0.$$

If $\alpha_1=\alpha_2$, we have $\gamma (\beta_1+\beta_2)=0$. Since $\gamma \neq 0$, we have $\beta_1=\beta_2$. From the equation in \eqref{EQ6.7}, we have $(\alpha_0+\gamma)\beta_0=0$, which implies $\beta_0=0$ or $\alpha_0=\gamma$. Since $\gamma\neq 0$, we have $\beta_0=0$. We have $\alpha_1=\alpha_2$, $\beta_1=\beta_2$, $\alpha_0\neq 0$, and $\beta_0=0$. This implies the out-degree of $e_0+1$ (when $\alpha_1=\alpha_2=0$) and $e_0+e_1+e_2+1$ (when $\alpha_1=\alpha_2=1$) is 1. 
\vskip 0.1in
Now assume that $\alpha_0=0$ and $\alpha_1 \neq \alpha_2$. Since $\alpha_0=0$ and $\gamma\neq 0$, from the first equation in \eqref{EQ6.7}, we have $\beta_0=0$. From the second equation in \eqref{EQ6.7}, we have $\alpha_1(\beta_1+\beta_2)+\beta_1=0$. If $\alpha_1=0$, $\beta_1=0$, and if $\alpha_1=1$, $\beta_2=0$. The two cases correspond to the zero-divisors $e_1+1$ and $e_2+1$, each of which has out-degree 1 because $\beta_0=0$ and $\beta_j=0$ for $j\neq 0$ in each case. 
\vskip 0.1in
Now we compute the in-degree of the zero-divisors $e_0+1$, $e_1+1$, $e_2+1$, $e_0+e_1+e_2+1$. 

Consider an arbitrary element $\alpha_0e_0+\alpha_1e_1+\alpha_2e_2$. 

$$(\alpha_0e_0+\alpha_1e_1+\alpha_2e_2)\cdot (e_0+1)=0$$
implies $\alpha_1=\alpha_2$. Thus there are $2^2-1$ many nonzero choices, which is equal to the out-degree of $(e_0+1)$, for the triple $(\alpha_0,\alpha_1,\alpha_1)$. 

$$(\alpha_0e_0+\alpha_1e_1+\alpha_2e_2)\cdot (e_j+1)=0,$$
where $j\neq 0$. This implies $2^3-1$ many nonzero choices, which is equal to the out-degree of $(e_j+1)$, where $j\neq 0$, for the triple $(\alpha_0,\alpha_1,\alpha_2)$. 

$$(\alpha_0e_0+\alpha_1e_1+\alpha_2e_2)\cdot (e_0+e_1+e_2+1)=0$$
implies $\alpha_1=\alpha_2$. Thus there are $2^2-1$ many nonzero choices, which is equal to the out-degree of $(e_0+e_1+e_2+1)$, for the triple $(\alpha_0,\alpha_1,\alpha_1)$. \vskip 0.1in

We have now computed all the zero-divisors and their in-degrees and out-degrees which are presented in the following graph and the table. 

\begin{center}
\begin{tikzpicture}[
    >=Latex,
    edge/.style={
    draw,
    thick,
    postaction={decorate},
    decoration={
        markings,
        mark=at position 0.50 with {\arrow{Latex}}
    }
}, 
]

% Vertices
\coordinate (A) at (-2,2);
\coordinate (B) at (-1,4);
\coordinate (C) at (1,5);
\coordinate (D) at (3,5);
\coordinate (E) at (5,4);
\coordinate (F) at (6,2);
\coordinate (G) at (6,0);
\coordinate (H) at (5,-1);
\coordinate (I) at (2,-2);
\coordinate (N) at (-1,-1);
\coordinate (O) at (-2,1);

% Solid points
\fill (A) circle (2.2pt);
\fill (B) circle (2.2pt);
\fill (C) circle (2.2pt);
\fill (D) circle (2.2pt);
\fill (E) circle (2.2pt);
\fill (F) circle (2.2pt);
\fill (G) circle (2.2pt);
\fill (H) circle (2.2pt);
\fill (I) circle (2.2pt);
\fill (N) circle (2.2pt);
\fill (O) circle (2.2pt);

% Labels
\node[left] at (A) {$x_0$};
\node[left] at (B) {$x_1$};
\node[left] at (C) {$x_2$};
\node[above right] at (D) {$x_0+1$};
\node[right] at (E) {$x_1+1$};
\node[right] at (F) {$x_2+1$};
\node[above right] at (G) {$x_0+x_1$};
\node[right] at (H) {$x_0+x_2$};
\node[left] at (I) {$x_1+x_2$};
\node[left] at (N) {$x_0+x_1+x_2$};
\node[left] at (O) {$x_0+x_1+x_2+1$};

\draw[edge] (A) -- (D);
\draw[edge] (D) to[bend right=15] (A);
\draw[edge] (A) -- (E);
\draw[edge] (A) -- (F);
\draw[edge] (A) -- (G);
\draw[edge] (A) -- (H);
\draw[edge] (A) -- (I);
\draw[edge] (A) -- (O);
% from and to x_1
\draw[edge] (B) -- (E);
\draw[edge] (E) to[bend left=15] (B);
\draw[edge] (B) -- (F);
\draw[edge] (B) -- (I);
% from and to x_2
\draw[edge] (C) --  (E);
%\draw[edge] (C) to[bend left=25]  (E);
\draw[edge] (C) -- (F);
\draw[edge] (C) -- (I);
\draw[edge] (F) to[bend right=15] (C);
%\draw[edge] (F) -- (C);
% to x_0+1
\draw[edge] (N) -- (D);
\draw[edge] (I) -- (D);
% from and to x_1+1
\draw[edge] (N) -- (E);
\draw[edge] (I) -- (E);
\draw[edge] (H) -- (E);
\draw[edge] (G) -- (E);
% from and to x_2+1
\draw[edge] (N) -- (F);
\draw[edge] (I) -- (F);
\draw[edge] (H) -- (F);
%\draw[edge] (H) to[bend right=85] (F);
\draw[edge] (G) -- (F);
% from and to x_0+x_1
\draw[edge] (N) -- (G);
\draw[edge] (I) to[bend left=15] (G);
%\draw[edge] (I) to[bend right=95] (G);
\draw[edge] (G) -- (I);
%\draw[edge] (G) to[bend left=45] (I);
% from and to x_0+x_2
\draw[edge] (N) -- (H);
\draw[edge] (I) -- (H);
\draw[edge] (H) to[bend left=25] (I);
% from and to x_1+x_2
\draw[edge] (N) -- (I);
\draw[edge] (I) -- (O);
\draw[edge] (I) to[out=-50,in=-150, looseness=10,out distance=2cm,in distance=2cm]  (I);
% from and to x_0+x_1+x_2
\draw[edge] (N) -- (O);
\draw[edge] (O) to[bend right=25] (N);
\end{tikzpicture}

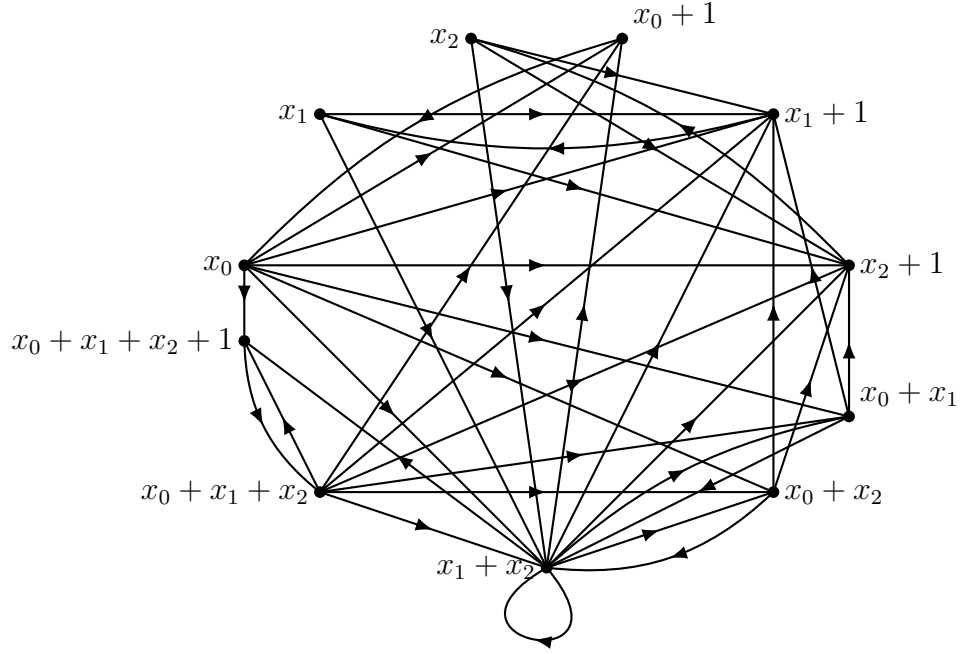
\captionof{figure}{Zero-divisor graph of the extended quandle ring of the Joyce quandle over $\mathbb{F}_2$}
\label{fig2}
\end{center}

\begin{table}[h]
    \centering
    %\caption{Sample Table Description}
    %\label{tab:sample_table3}
    \begin{tabular}{|c|c|c|}
        \hline
        Vertex & In-degree & Out-degree \\
        \hline
        $e_0$ & 1 & $2^3-1$ \\
        $e_j$, $j\neq 0$  & 1 & $2^2-1$ \\
          $e_0+1$  & $2^2-1$ & 1 \\
           $e_j+1$, $j\neq 0$  & $2^3-1$ & 1 \\
        $e_0+e_j$, $j\neq 0$  & $2^2-1$ & $2^2-1$ \\
         $e_1+e_2$  & $2^3-1$ & $2^3-1$ \\
         $e_0+e_1+e_2$  & 1 & $2^3-1$ \\
          $e_0+e_1+e_2+1$  & $2^2-1$ & 1 \\
            \hline
  \end{tabular}
  \captionof{table}{In-degree and out-degree table of Figure~\ref{fig2}}
\label{tab2}
\end{table}
%%%%%%% 

%\begin{minipage}{0.48\textwidth}

\vskip 0.1in 

It is clear from Table~\ref{tab2} that, as in the case of the quandle ring of the trivial quandle, the sum of two zero-divisors with mirror symmetry is not a zero-divisor.
\vskip 0.1in
\subsection{Zero-divisor graph of the extended quandle ring of the dihedral quandle}
There are three quandles of order 3 up to isomorphism, and among them, there is only one commutative quandle, called the dihedral quandle. The zero-divisor graph of the extended quandle ring of the dihedral quandle over $\mathbb{F}_2$ is an example of an undirected graph as the ring is commutative. 
 %\textcolor{red}{include the graph}
\begin{center}
\begin{tikzpicture}[
    >=Latex,
    edge/.style={
    draw,
    },
]

% Vertices
\coordinate (A) at (-1,2); %x_0
\coordinate (B) at (1,2); %x_1
\coordinate (C) at (2,0);%x_2
\coordinate (D) at (-2,2);%x_0+1
\coordinate (E) at (0,2); %x_1+1
\coordinate (F) at (2,1);%x_2+1
\coordinate (G) at (2,-1); %x_0+x_2
\coordinate (H) at (0,-2);%x_0+x_1
\coordinate (I) at (-2,-1); %x_1+x_2
\coordinate (J) at (1,-2);%x_0+x_2+1
\coordinate (K) at (-1,-2);%x_0+x_1+1
\coordinate (L) at (-3,-1);%{$x_1+x_2+1$}
\coordinate (M) at (0,0); %{$x_0+x_1+x_2$}
\coordinate (N) at (-3,1);%{$x_0+x_1+x_2+1$}

% Solid points
\fill (A) circle (2.2pt);
\fill (B) circle (2.2pt);
\fill (C) circle (2.2pt);
\fill (D) circle (2.2pt);
\fill (E) circle (2.2pt);
\fill (F) circle (2.2pt);
\fill (G) circle (2.2pt);
\fill (H) circle (2.2pt);
\fill (I) circle (2.2pt);
\fill (J) circle (2.2pt);
\fill (K) circle (2.2pt);
\fill (L) circle (2.2pt);
\fill (M) circle (2.2pt);
\fill (N) circle (2.2pt);

% Labels
\node[above] at (A) {$x_0$};
\node[right] at (B) {$x_1$};
\node[right] at (C) {$x_2$};
\node[left] at (D) {$x_0+1$};
\node[above] at (E) {$x_1+1$};
\node[above] at (F) {$x_2+1$};
\node[above right] at (G) {$x_0+x_2$};
\node[below] at (H) {$x_0+x_1$};
\node[below] at (I) {$x_1+x_2$};
\node[right] at (J) {$x_0+x_2+1$};
\node[left] at (K) {$x_0+x_1+1$};
\node[left] at (L) {$x_1+x_2+1$};
\node[left] at (M) {$x_0+x_1+x_2$};
\node[left] at (N) {$x_0+x_1+x_2+1$};

\draw[edge] (M) -- (D);
\draw[edge] (D) -- (A);
\draw[edge] (M) -- (E);
\draw[edge] (E) -- (B);
\draw[edge] (M) -- (F);
\draw[edge] (F) -- (C);
\draw[edge] (M) -- (G);
\draw[edge] (G) -- (J);
\draw[edge] (M) -- (H);
\draw[edge] (H) -- (K);
\draw[edge] (M) -- (I);
\draw[edge] (I) -- (L);
\draw[edge] (M) -- (N);
%\draw[edge] (I) -- (D);

\end{tikzpicture}

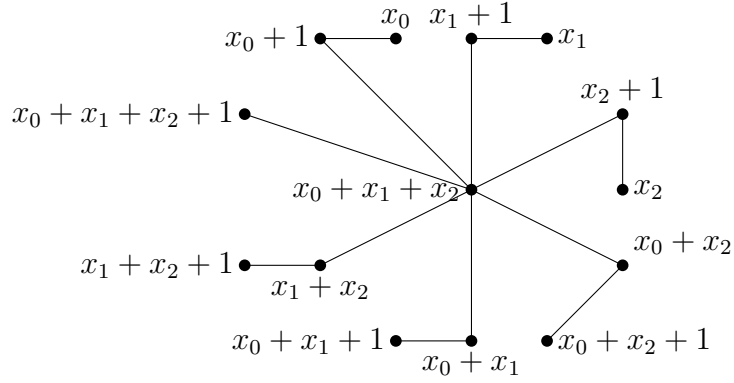
\captionof{figure}{Zero-divisor graph of the extended quandle ring of the dihedral quandle over $\mathbb{F}_2$}
\label{fig4}
\end{center}
\vskip 0.1cm
From the graph, it is evident that every element of the extended dihedral quandle ring over $\mathbb{F}_2$, other than $0$ and the unity, is a zero divisor. It is also observed that any vertex $x$ with $\epsilon^o(x)=0$ has degree $2.$ On the other hand, every vertex with $\epsilon^o(x)=1$ has degree $1$, with the unique exception of the vertex $x_0+x_1+x_2$, which has degree $7$. 

\vskip 0.1in 

We believe that the results in this section would motivate the quandle ring enthusiasts for a further investigation on zero-divisor graphs of quandle rings. 

\section{Prime rings and semi-prime rings}\label{S7}

A ring $\Bbbk$ is prime if the product of any two nonzero ideals of $\Bbbk$ is nonzero. A ring $\Bbbk$ is semi-prime if the only ideal of $\Bbbk$ which squares to zero is the zero ideal. Every prime ring is semi-prime, but the converse is not true in general. In this section, we explain a case in which the quandle ring $\mathbb{Z}_p[Q]$, where $p$ is a prime number, is not semi-prime, and therefore not prime. 

\begin{prop}\label{P1}
Let $p$ be prime and $Q$ be a quandle of order $n\geq 2$. If $p\,|\,n$, then the quandle ring $\mathbb{Z}_p[Q]$ is not semi-prime, and therefore not prime.     
\end{prop}

\begin{proof}
Let $Q=\{x_0,x_1,\ldots, x_{n-1}\}$. Consider the ideal generated by $x_0+x_1+\cdots +x_{n-1}$: $\langle x_0+x_1+\cdots + x_{n-1}\rangle$.\vskip 0.1in 

\noindent Clearly, it is straightforward to see that 
\[
\begin{split}
&(x_0+x_1+\cdots +x_{n-1})\cdot (\alpha_0x_0+\alpha_1x_1+\cdots + \alpha_{n-1}x_{n-1})\cr
&=(\alpha_0+\alpha_1+\cdots +\alpha_{n-1})(x_0+x_1+\cdots +x_{n-1}).
\end{split}
\]

\noindent This means, in $\mathbb{Z}_p[Q]$, we have 

$$I=\langle x_0+x_1+\cdots + x_{n-1}\rangle=\{\alpha\cdot (x_0+x_1+\cdots +x_{n-1})\,|\,\alpha\in \mathbb{Z}_p\}.$$

\noindent Now consider $(\alpha\cdot (x_0+x_1+\cdots +x_{n-1}))^2$. We have 

$$(\alpha\cdot (x_0+x_1+\cdots +x_{n-1}))^2=n\cdot \alpha^2\cdot (x_0+x_1+\cdots +x_{n-1}).$$

\noindent Since $p\,|\,n$, $I^2=0$. We have shown that the quandle ring $\mathbb{Z}_p[Q]$ is not semi-prime, and therefore the ring is not prime. 

\end{proof}

We obtain the following immediate corollary which also appeared in \cite{BES-2026}. 

\begin{cor}
Let $p$ be prime and $Q$ be a quandle of order $n\geq 2$. Then $$\displaystyle{\frac{(x_0+x_1+\cdots +x_{n-1})}{n}}$$ is an idempotent in the quandle ring $\mathbb{Z}_p[Q]$.
\end{cor}

\begin{rmk}
We can generalize the result in Proposition~\ref{P1}. Let $m>1$ be an integer and $Q$ a quandle of order $n\geq 2$. If $m\,|\,n$ or $m\,|\,\alpha^2$, then the quandle ring $\mathbb{Z}_m[Q]$ is not semi-prime, and therefore not prime.
\end{rmk}

\section{Polynomials and commutative quandles}\label{S8}

Given a quandle ring, it is natural to investigate the ring of polynomials whose coefficients are in the given quandle ring, and their applications.  Let $n\geq 1$ be an integer such that $2n+1$ is prime. In this section, we find a bivariate polynomial in the ring $(\mathbb{Z}_{2n+1}[Q])[X,Y]$ that defines the commutative symmetric quandle of order $2n+1$. We recall that a commutative quandle $Q$ of order $2n+1$ is a quandle whose operation is given by $r\ast s=(n+1)(r+s)$ for all $r,s,\in Q$.   \vskip 0.1in 

\noindent We apply bivariate Lagrange Interpolation to find a bivariate polynomial $f(X,Y)\in (\mathbb{Z}_{2n+1}[Q])[X,Y]$ such that $f(a,b)=e_{(n+1)(a+b)\pmod{2n+1}}$. \vskip 0.1in 

\noindent $\displaystyle{L_{(X,0)}(X)=\frac{1}{(-1)^{2n}\,0!\,(2n)!}\,\prod_{\substack{i=0\\i\neq 0}}^{2n}\,(X-i)}$\vskip 0.2in
\noindent $\displaystyle{L_{(X,1)}(X)=\frac{1}{(-1)^{2n-1}\,1!\,(2n-1)!}\,\prod_{\substack{i=0\\i\neq 1}}^{2n}\,(X-i)}$\vskip 0.2in
\noindent $\displaystyle{L_{(X,2)}(X)=\frac{1}{(-1)^{2n-2}\,2!\,(2n-2)!}\,\prod_{\substack{i=0\\i\neq 2}}^{2n}\,(X-i)}$\vskip 0.2in
\noindent $\cdots \cdots \cdots \cdots \cdots \cdots$ \vskip 0.2in
\noindent $\cdots \cdots \cdots \cdots \cdots \cdots$ \vskip 0.2in
\noindent $\displaystyle{L_{(X,2n)}(X)=\frac{1}{(-1)^{0}\,(2n)!\,0!}\,\prod_{\substack{i=0\\i\neq n-1}}^{2n}\,(X-i)}$\vskip 0.2in

\noindent We can write polynomials $L_{(X,\cdot)}(X)$ as follows:\vskip 0.2in 

\noindent \noindent $\displaystyle{L_{(X,a)}(X)=\frac{1}{(-1)^{2n-a}\,a!\,(2n-a)!}\,\prod_{\substack{i=0\\i\neq a}}^{2n}\,(X-i)}$\,\,\,\,\,for $0\leq a\leq 2n$. \vskip 0.2in

\noindent Because of symmetry we have 

\noindent \noindent $\displaystyle{L_{(Y,b)}(Y)=\frac{1}{(-1)^{2n-b}\,b!\,(2n-b)!}\,\prod_{\substack{j=0\\j\neq b}}^{2n}\,(Y-j)}$\,\,\,\,\,for $0\leq b\leq 2n$. \vskip 0.2in

Note that $(-1)^{2n-a}\,a!\,(2n-a)!=-1$. 

\vskip 0.2in

\noindent The polynomial $$\displaystyle{f(X,Y)=\sum_{a=0}^{2n}\,\sum_{b=0}^{2n}\, e_{(n+1)(a+b)\pmod{2n+1}}\cdot L_{(X,a)}(X)\,\cdot L_{(Y,b)}(Y)}\,\,\in \,\,(\mathbb{Z}_{2n+1}[Q])[X,Y]$$
is the desired polynomial. 

\vskip 0.1in

\begin{exmp}

\noindent Let $Q$ be the dihedral quandle of order 3 whose multiplication table is given below. 

\[
\begin{tabular}{c | c c c}
$\ast$ & $e_0$ & $e_1$ & $e_2$  \\
\cline{1-4}
$e_0$& $e_0$ & $e_2$ & $e_1$ \\
$e_1$& $e_2$ & $e_1$ & $e_0$ \\
$e_2$& $e_1$ & $e_0$ & $e_2$ \\
\end{tabular}
\]

\vskip 0.1in  

\noindent We need to find a polynomial $f(X,Y)$ such that 
$$f(0,0)=e_0,\hspace{0.2cm}f(0,1)=e_2,\hspace{0.2cm}f(0,2)=e_1$$
$$f(1,0)=e_2,\hspace{0.2cm}f(1,1)=e_1,\hspace{0.2cm}f(1,2)=e_0$$
$$f(2,0)=e_1,\hspace{0.2cm}f(2,1)=e_0,\hspace{0.2cm}f(2,2)=e_2$$

\vskip 0.2in 

\noindent We apply bivariate Lagrange Interpolation. \vskip 0.1in 

\noindent $\displaystyle{L_{(X,0)}(X)=\frac{(X-1)(X-2)}{(0-1)(0-2)}=\frac{1}{2}(X-1)(X-2)}$\vskip 0.2in
\noindent $\displaystyle{L_{(X,1)}(X)=\frac{(X-0)(X-2)}{(1-0)(1-2)}=-X(X-2)}$\vskip 0.2in
\noindent $\displaystyle{L_{(X,2)}(X)=\frac{(X-0)(X-1)}{(2-0)(2-1)}=\frac{1}{2}X(X-1)}$\vskip 0.2in

\noindent Because of symmetry, we have \vskip 0.1in

\noindent $\displaystyle{L_{(Y,0)}(Y)=\frac{(Y-1)(Y-2)}{(0-1)(0-2)}=\frac{1}{2}(Y-1)(Y-2)}$\vskip 0.2in
\noindent $\displaystyle{L_{(Y,1)}(Y)=\frac{(Y-0)(Y-2)}{(1-0)(1-2)}=-Y(Y-2)}$\vskip 0.2in
\noindent $\displaystyle{L_{(Y,2)}(Y)=\frac{(Y-0)(Y-1)}{(2-0)(2-1)}=\frac{1}{2}Y(Y-1)}$\vskip 0.2in

\[
\begin{split}
f(X,Y)&=   L_{(X,0)}(X)\Big(L_{(Y,0)}(Y)\cdot e_0+L_{(Y,1)}(Y)\cdot e_2+L_{(Y,2)}(Y)\cdot e_1\Big)\cr 
&+L_{(X,1)}(X)\Big(L_{(Y,0)}(Y)\cdot e_2+L_{(Y,1)}(Y)\cdot e_1+L_{(Y,2)}(Y)\cdot e_0\Big)\cr
&+L_{(X,2)}(X)\Big(L_{(Y,0)}(Y)\cdot e_1+L_{(Y,1)}(Y)\cdot e_0+L_{(Y,2)}(Y)\cdot e_2\Big)\cr
\end{split}
\]

\vskip 0.1in 

\[
\begin{split}
f(X,Y)&=  \frac{1}{2}(X-1)(X-2)\Big(\frac{1}{2}(Y-1)(Y-2)\cdot e_0-Y(Y-2)\cdot e_2+\frac{1}{2}Y(Y-1)\cdot e_1\Big)\cr 
&-X(X-2)\Big(\frac{1}{2}(Y-1)(Y-2)\cdot e_2-Y(Y-2)\cdot e_1+\frac{1}{2}Y(Y-1)\cdot e_0\Big)\cr 
&+\frac{1}{2}X(X-1)\Big(\frac{1}{2}(Y-1)(Y-2)\cdot e_1-Y(Y-2)\cdot e_0+\frac{1}{2}Y(Y-1)\cdot e_2\Big),
\end{split}
\]

\vskip 0.2in 

\noindent which is equivalent to 

\[
\begin{split}
f(X,Y)&=(2e_0+2e_1+2e_2)\,X^2+(2e_0+2e_1+2e_2)\,Y^2\cr
&+(e_0+e_1+e_2)\,XY\,+(e_1+2e_2)\,X+(e_1+2e_2)\,Y+e_0
\end{split}
\]

\noindent The polynomial $f(X,Y)$ is a symmetric polynomial. It is easy to check in characteristic three that 

$$f(0,0)=e_0,\hspace{0.2cm}f(0,1)=e_2,\hspace{0.2cm}f(0,2)=e_1$$
$$f(1,0)=e_2,\hspace{0.2cm}f(1,1)=e_1,\hspace{0.2cm}f(1,2)=e_0$$
$$f(2,0)=e_1,\hspace{0.2cm}f(2,1)=e_0,\hspace{0.2cm}f(2,2)=e_2$$

\end{exmp}

\section*{Acknowledgments}
Neranga Fernando is grateful to Allen Broughton, Daniel J. Katz and Mohamed Elhamdadi for the valuable discussions. Bhitali Kousik is grateful to her supervisor Prof. Dhiren Kumar Basnet at Tezpur University for his support in pursuing this collaboration and was financially supported by DST-INSPIRE Fellowship, Government of India (INSPIRE Reg. No. IF230368). Indu Rasika Churchill gratefully acknowledges SUNY Oswego for granting the sabbatical leave during which this research was conducted. The authors are also grateful to Zhi-lin Zhang for the valuable comments on the first preprint.

\end{document}